\documentclass{article}

\usepackage{amsmath}
\usepackage{amsthm}
\usepackage{amsfonts}
\usepackage{tikz}
\usetikzlibrary{arrows.meta, positioning,calc}
\usepackage{pgfplots}
\pgfplotsset{compat=1.18}
\usepackage[colorlinks,linkcolor=blue,anchorcolor=blue,citecolor=blue]{hyperref}
\usepackage{geometry}
\usepackage{amsfonts}
\usepackage{graphicx}
\usepackage{epstopdf}
\usepackage[ruled,linesnumbered]{algorithm2e}
\usepackage{algorithmic}
\usepackage{bm}
\usepackage{enumerate}
\usepackage{booktabs}
\usepackage{multirow}
\usepackage{mathabx}
\usepackage{graphics}
\usepackage[sort,numbers]{natbib}
\usepackage{color}

\usepackage{caption}

\definecolor{orange}{RGB}{255,90,0}
\definecolor{lightgray}{gray}{0.7}

\numberwithin{equation}{section}
\numberwithin{figure}{section}

\newcommand{\zl}[1]{\mathcal{#1}}
\newcommand{\jz}[1]{\mathbf{#1}}
\newcommand{\xl}[1]{\mathbf{#1}}
\newcommand{\quat}[1]{\mathbb{H}^{#1}}

\newcommand{\bigxiaokuohao}[1]{\ensuremath{ \left(  #1 \right) }}

\newcommand{\fnorm}[1]{\ensuremath{ \|   #1  \|_F }} 

\newcommand{\refsect}[1]{Sect.~\ref{#1}}
\newcommand{\refdef}[1]{Definition~\ref{#1}}
\newcommand{\refthm}[1]{Theorem~\ref{#1}}

\newcommand{\refeq}[1]{\eqref{#1}}

\newcommand{\refalgo}[1]{Algorithm~\ref{#1}}

\newenvironment{mytabular1}{\bgroup\footnotesize\tabular}{\endtabular\egroup}

\newtheorem{example}{Example}
\newtheorem{remark}{Remark}
\newtheorem{observation}{Observation}
\usepackage{amsopn}

\newtheorem{proposition}{Proposition}
\newtheorem{theorem}{Theorem}
\newtheorem{definition}{Definition}
\newtheorem{lemma}{Lemma}

\title{A Parallelizable Quaternion Higher-Order Singular Value Decomposition with Applications}
\author{Hanxin Ya$^*$,~  Yuning Yang\thanks{College of Mathematics and Information Science, Guangxi University, Nanning, 530004, China}~\thanks{Center for Applied Mathematics of Guangxi, Guangxi University, Nanning, 530004, China.} \thanks{Corresponding author: Yuning Yang, yyang@gxu.edu.cn} }

\begin{document}
\maketitle

\begin{abstract}
 Higher-order singular value decomposition (HOSVD) is a celebrated tool for tensor data analysis.    The sequential   HOSVD was recently  generalized to the quaternion domain, while a naive quaternion extension of the classical HOSVD
 , which can be excecuted in parallel, incurs   issues. To leverage the power of   parallel computing,  this work introduces a two-sided quaternion HOSVD (TS-QHOSVD) that can be parallelized on two processors. It is proved that TS-QHOSVD    (i) preserves the   HOSVD ordering property, (ii) inherits the orthogonality property     at the first and the last modes,  and (iii) satisfies the weak orthogonality     at all modes.  The truncated TS-QHOSVD is then developed, 
with its error bound   being established.   
  We apply the proposed model on color video denoising  as well as scientific data   compression arising from 3D Navier-Stokes equation and Lorentz system to demonstrate its efficacy.

\end{abstract}

  

\section{Introduction}
{\color{black}In the big-data era, encoding multi-dimensional and multi-modal data into a tensor-based format allows one to fully exploit  the inherent  features.}
Several tensor decompositions have been proposed  \cite{kolda2009tensor,de2008decompositions,zhao2016tensor,oseledets2011tensor,kilmer2011factorization}. 
The   higher-order singular value decomposition \cite{hosvd} (HOSVD)     factorizes the unfolding matrices of the data tensor along each mode simultaneously, while  the sequential HOSVD  \cite{vannieuwenhoven2012new}   generates factor matrices sequentially from    tensors that are   the products of the data tensor and the previously computed factor matrices, which may be more efficient on a single processor.

Quaternions generalize complex numbers to four dimensions. Although      their multiplications are non-commutative,  their ability to represent 3D and 4D vectors as   single quaternion scalars  led to a growing interest in the study, especially in image and video processing,  such as image fusion and denoising \cite{miao2023quat}, face recognition \cite{zou2016quat,zou2021quat},  color video approximation \cite{miao2020low,qin2022sigular,miao2023qt-svd}, and   inpainting \cite{jia2019robust,miao2020quaternion,jia2022non,chen2022color}. \cite{miron2023quaternions} gives    comprehensive survey of applications in signal and image processing recentlly. Quaternion matrix optimization also draws much attention \cite{xu2015optimization,qi2022quaternion}. 
The adaptation of matrix computations to the quaternion domain has been a key factor in the success of quaternions   \cite{zhang1997quaternions,wei2018quaternion,jia2018new,jia2019lanczos,jia2021structure}. 
Driven by the representation of multi-dimensional data as quaternion tensors, 
  several tensor algebras have also been  extended to the quaternion case, such as the
 T-product \cite{qin2022sigular,miao2023qt-svd,zheng2023approximation}, 
  tensor-train decomposition \cite{miao2022quat},      mode-$k$ product \cite{miao2020low,miao2023quat}, and Einstein product \cite{he2023eigenvalues}.    Extensions of Tucker decomposition and CP decomposition  to  the quaternion domain were studied in  \cite{osimone2023low,flamant2024multilinear}.

  Recently, Miao et al. \cite{miao2023quat} proposed quaternion HOSVD  in the fashion of the sequential HOSVD \cite{vannieuwenhoven2012new}, and successively applied it to   image fusion and denoising. They proved the existence of such a  decomposition  and the ordering property.   This quaternion HOSVD is based on left multiplications (and due to this reason, we term   it as L-QHOSVD), namely, the   data tensor is decomposed as a core tensor  left  multiplied by   factor matrices along each mode. Note that in the quaternion domain,  tensor mode-$k$ products are not commutative. 

For large tensors, modern parallel   architechtures are helpful for accelerating computations. 
The model of \cite{miao2023quat} cannot be executed in parrallel.  The original HOSVD \cite{hosvd}  that computes the factors in parallel  might be naively generalized to   quaternions. However, unlike its real/complex counterpart \cite[Property 10]{hosvd},    the reconstruction error of the truncated HOSVD may not be controlled by the tail energy (see Example \ref{example:1}), which means that a low-rank compression of the quaternion data tensor using HOSVD is unstable. In addition, some theoretical properties of HOSVD do not hold in the quaternion domain, making the decomposition  less interpretable.  

To address this limitation, this work   proposes a two-sided quaternion HOSVD (TS-QHOSVD) that can be parallelly excecuted on two processors. 
 Roughly speaking,  given a $2m$-th order tensor,    for  Processor 1, the tensor performs mode-$k$ left products for the first $m$ modes  sequentially, for which the sequential HOSVD mechanism is executed. In parallel, for     Processor 2, similar tasks are performed for the remaining modes, except that right mode-$k$ products are performed. To achieve this, left and right mode-$k$ unfoldings, tensor-matrix products, and necessary properties, are introduced. 
We then  show that
  TS-QHOSVD   admits the ordering property as HOSVD, while the orthogonality property   holds only at the first and the last modes. As an   alternative, we define the weak orthogonality and show that it  still holds at all modes. 
We then develope truncated TS-QHOSVD and establish its      error bound.  We provide numerical examples on color video denoising, compression, and scientific data tasks to demonstrate the efficiency and effectiveness of the proposed model. In particular, we represent the data computed from a 3D Navier-Stokes equation  for microscopic natural
convection as a quaternion tensor ($18.69~{\rm GB}$ in memory, $4.86$GB in disk), and apply truncated TS-QHOSVD to compress it.  With only $12.0$MB compressed data, the reconstructed tensor well approximates the original one.

{\color{black}The rest is organized as follows.   \refsect{pre}  prepares    preliminaries on quaternions.
  \refsect{tensor} introduces  left and right mode-$k$ unfolding, mode-$k$ products, and some properties for quaternion tensors.
  \refsect{decomposition} presents TS-QHOSVD and   its properties. Then, the truncated TS-QHOSVD with its error bounds is investigated in \refsect{error}. Numerical experiments are provided in
  \refsect{examples}.
Finally, some conclusions are drawn in \refsect{conclusions}.}

\section{Preliminaries}\label{pre}

\paragraph{Notations}
  $\mathbb{R}, \mathbb{C}$, and $\mathbb{H}$ respectively denote the real, complex, and quaternion space. $  \quat{I_1\times I_2\times \cdots \times I_N}$ represents the $N$-th order quaternion tensor space.
A scalar, vector, matrix, and tensor are written as $a$, $\xl{a}$, $\jz{A}$, and $\zl{A}$, respectively.
  $\bar{(\cdot)}$, $(\cdot)^T$, and $(\cdot)^H$ respectively represent the conjugate, transpose, and conjugate transpose.
  $\otimes$ denotes the Kronecker product. $tr(\cdot)$ represents the trace of a matrix.

Throughout this work, the $(i,j)$-th entry of a matrix $\jz A$ is written as $a_{ij}$, and the $(i_1,\ldots,i_N)$-th entry of a tensor $\zl A$ is written as $a_{i_1\cdots i_N}$. At the same time,  it may be more convenient to use the MATLAB notations $\jz A(i,j)$ and $\zl A(i_1,\ldots,i_N)$ to denote the same entries in some contexts. Correspondingly,  $\jz A(i,:)$ and $\jz A(:,j)$ respectively represent the $i$-th row and the $j$-th column of $\jz A$.

A four-dimensional non-commutative division algebra $\mathbb{H}$ is defined over the real numbers $\mathbb{R}$ with 
canonical basis $\{1,\bf i,j,k\}$. Here ${\bf i},{\bf j},{\bf k}$ are imaginary units satisfying
\[\bf{i}^2=\bf{j}^2=\bf{k}^2={\bf ijk}=-1,~{\rm and}~\bf{ij}=-\bf{ji}=k,\bf{jk}=-\bf{kj}=\bf{i},\bf{ki}=-\bf{ik}=\bf{j}.\]
Any quaternion $q\in \mathbb{H}$ can be written as 
 $q=q_a+q_b{\bf i}+ { q_c} {\bf j}+q_d{\bf k}$,  
where $q_a,q_b,q_c,q_d\in \mathbb{R}$ are the components of $q$.
The real and imaginary parts of $q$ are denoted as ${ Re}(q)=q_a$ and ${ Im}(q)=q_b{\bf i}+q_c{\bf j}+q_d{\bf k}$ respectively.
For any $p,q\in \mathbb{H}$,   $pq\neq qp$ in general.
  $\bar{q}=q_a-q_b{\bf i}-q_c{\bf j}-q_d{\bf k}$ is denoted as the quaternion conjugate of $q$, and it holds that $\overline{p}\overline{q}=\bar{q}\bar{p}$.
For   $q\in \mathbb{H}$, its modulus is defined as 
$|q|=\sqrt{q\bar{q}}=\sqrt{\bar{q}q}=\sqrt{q_a^2+q_b^2+q_c^2+q_d^2}$. A quaternion vector/matrix is a vector/matrix with quaternion entries.

\begin{proposition}[c.f. \cite{zhang1997quaternions}]
    For given two quaternion  matrices $\jz{P}\in \quat{M\times N}$ and $\jz{Q}\in \quat{N\times M}$, 
\begin{itemize}
    \item $(\jz{P}\jz{Q})^H=\jz{Q}^H\jz{P}^H$; $(\jz{P}\jz{Q})^T\neq \jz{Q}^T\jz{P}^T$ in general; $\overline{\jz{P}\jz{Q}}\neq  \bar{\jz{P}}  \bar{\jz{Q}} $ in general.
\end{itemize}
\end{proposition}

Although $(\jz P\jz Q)^T \neq \jz Q^T\jz P^T$ for two general quaternion matrices, if one of them is real, then as  real numbers and quaternions   commute, this relation is still valid:
\begin{proposition}
    \label{lem:real_quater_transpose_commutat}
    Assume that $\jz P\in \mathbb R^{M\times N}$ and $\jz Q\in\quat{N\times M}$
    . Then $(\jz P\jz Q)^T = \jz Q^T\jz P^T$.
\end{proposition}

\color{black}
It is also easy to see that:
\begin{proposition}
    \label{lem:conj_real_conterpart}
    For any $\jz{A}\in \quat{M\times N}$ and $\jz{B}\in \quat{N\times M}$,  one has $Re(\jz{A}\jz{B})= Re(\widebar{\jz A\jz B})= Re(\bar{\jz{A}}\bar{\jz{B}})$.
\end{proposition}

\color{black}

A square quaternion matrix $\jz U$ is called unitary if $\jz U^H\jz U = \jz{I}$, where $\jz{I}$ is the identity matrix of a proper size. For unitary matrices, similar to the real case,   the following holds:
\begin{proposition}
    \label{lem:kron_unitary}
    Let $\jz U,\jz V$ be two unitary quaternion matrices. Then their Kronecker product $\jz U\otimes \jz V$ is still unitary.
\end{proposition}

However, if $\jz U$ is a unitary quaternion matrix, then its transpose $\jz U^T$ may not be. 

  Singular value decomposition of  a quaternion matrix  is given as follows:
\begin{theorem}[Quaternion SVD \cite{zhang1997quaternions}]
    Let $\jz{M}\in \quat{I_1\times I_2}$ ($I_1\leq I_2$) be of rank $r$. Then  there exist a unitary quaternion matrix
    $\jz{U}\in \quat{I_1\times I_1}$ and a partially orthonormal $\jz{V}\in \quat{I_2\times I_1}$ with $\jz{V}^H\jz{V}=\jz{I}$ such that 
     $  \jz{M} =\jz{U}\Sigma\jz{V}^H  $
    where $\Sigma=diag(\sigma_1,\dots,\sigma_r,\overbrace{0,\ldots,0}^{I_1-r})$ is a   diagonal matrix with   $\sigma_k$'s being positive singular values of   $\jz{M}$.
\end{theorem}

   Quaternion tensors are defined as follows. 
\begin{definition}[Quaternion tensors]
    An $N$-th order tensor is  called a quaternion tensor  if its entries are quaternions, i.e.,
    \[\zl{T}=(t_{i_1i_2\cdots i_N})\in \quat{I_1\times I_2\times \cdots \times I_N}=\zl{T}_a+\zl{T}_b\bf{i}+\zl{T}_c\bf{j}+\zl{T}_d\bf{k} \]
    with $\zl{T}_a,\zl{T}_b,\zl{T}_c,\zl{T}_d\in \mathbb{R}^{I_1\times I_2\times \cdots \times I_N}$. $\zl{T}$ is a pure tensor if $\zl{T}_a$ is zero.
\end{definition}
We     present the   definitions of left and right inner product and orthogonality for quaternion tensors.

\begin{definition}[Inner product] \label{def:left_right_inner_prod}
    The left and right inner products of two quaternion tensors $\zl{A},\zl{B}\in\quat{I_1\times I_2\times \cdots \times I_N}$ are respectively defined as
    \begin{align*}
       & \langle \zl{A} , \zl{B} \rangle_L= \sum_{i_1=1}^{I_1}\sum_{i_2=1}^{I_2}\cdots \sum_{i_N=1}^{I_N} a_{i_1i_2\cdots i_N}\bar{b}_{i_1i_2\cdots i_N} ; ~~~~
        \langle \zl{A} , \zl{B} \rangle_R= \sum_{i_1=1}^{I_1}\sum_{i_2=1}^{I_2}\cdots \sum_{i_N=1}^{I_N} \bar{a}_{i_1i_2\cdots i_N}b_{i_1i_2\cdots i_N}.
    \end{align*}
\end{definition}
Here,  the right inner product coincides with the (standard) quaternion matrix inner product;  see, e.g., \cite{chen2022color,qi2022quaternion}, and its   conjugation coincides with the quaternion vector inner product  \cite[Def. 3.1.2]{rodman2014topics}.    
It is not hard to see that $\widebar{\langle \zl A,\zl B\rangle_L} = \langle \zl B,\zl A\rangle_L$ and $\widebar{\langle \zl A,\zl B\rangle_R} = \langle \zl B,\zl A\rangle_R$.

\begin{definition}[Orthogonality]\label{def:orthogonality}
 For $\zl{A},\zl{B}\in\quat{I_1\times I_2\times \cdots \times I_N}$, if $ \langle \zl{A} , \zl{B} \rangle_L=0$, then we say that   they are left orthogonal; if $\langle \zl{A} , \zl{B} \rangle_R=0$, then we say that they are right orthogonal; if $  Re\langle \zl{A} , \zl{B} \rangle_R =0$, then we say that they are weakly orthogonal. 
\end{definition}
The right orthogonality coincides with the notion of the standard orthogonality of   quaternion vectors \cite{zhang1997quaternions,rodman2014topics} in the usual sense. On the other hand, one always has for any $\zl A,\zl B$ of the same order, 
\begin{align}\label{eq:real_inn_prod_equal}Re\langle \zl{A} , \zl{B} \rangle_L = Re\langle \zl{A} , \zl{B} \rangle_R = Re\langle \zl{B} , \zl{A} \rangle_L = Re\langle \zl{B} , \zl{A} \rangle_R.\end{align}
Thus if any of them is zero, then $\zl A$ and $\zl B$ are weakly orthogonal.

     The Frobenius norm of quaternion tensors   is given by 
     $\|\zl{A}\|_F^2=\langle \zl{A}, \zl{A} \rangle_L=\langle \zl{A}, \zl{A} \rangle_R  $.

\section{Left and Right Tensor-Matrix Basic Operations}\label{tensor}
As   quaternion multiplications are non-commutative, several   concepts  such as eigenvalues, matrix products, and Fourier transforms,   have   left and right counterparts \cite{zhang1997quaternions,schulz2014using,miron2023quaternions}. 
In this section, for quaternion tensor-matrix operations, we   introduce  their left/right   unfoldings,       products, and some necessary   properties for later use. 

\subsection{Left and right mode-$k$ unfoldings}

The mode-$k$ unfolding of a real tensor is given by arranging the mode-$k$ fibers to the columns of the resulting matrix \cite{kolda2009tensor}. This definition has  been   extended to  left mode-$k$ unfolding of a quaternion tensor  \cite{miao2020low} without the prefix     ``left''.  Left and right tensor mode-$k$ unfolding are given explicitly as follows.
\begin{definition}[Left and right mode-$k$ unfoldings]
    For   $\zl{T}\in \quat{I_1\times I_2\times \cdots \times I_N}$, its left mode-$k$ unfolding is denoted as $\jz{T}^L_{[k]}$  
    by arranging the mode-$k$ fibers to   the columns of the resulting quaternion matrix, i.e.,
    \[\jz{T}^L_{[k]}\in \quat{I_k \times I_1\cdots I_{k-1}I_{k+1}\cdots I_{N}}.\]
    The $(i_1,i_2,\dots,i_N)$-th entry of $\zl{T}$ maps to the  $(i_k,j)$-th entry of $\jz{T}^L_{[k]}$, where
    \begin{equation}\label{eq:def:l_unfolding:1}j=1+\sum_{\substack{l=1 \\l\neq k}}^{N}\nolimits(i_l-1)J_l\qquad \text{with}\qquad J_l=
    \prod \nolimits_{\substack{m=1 \\ m\neq k}}^{l-1}I_m.\end{equation}
The right mode-$k$ unfolding of $\zl T$, denoted as $\jz{T}^R_{[k]}$, is given by   the transpose of $\jz{T}^L_{[k]}$:
\[
 \jz{T}^R_{[k]}:=(\jz{T}^L_{[k]})^T\in \quat{I_1\cdots I_{k-1}I_{k+1}\cdots I_{N}\times I_k}. 
\]
    
    \label{lmu}
\end{definition}
 
The left unfolding coincides with the real-case unfolding. On the other hand, although the right unfolding is merely the transpose of the left one, keeping them distinct simplifies subsequent computation and analysis.


 For a quaternion matrix $\jz A$, it holds that $\jz A= \jz A^L_{[1]} = \jz A^R_{[2]}$.

\subsection{Left and right  mode-$k$ products}

Recall that for a real tensor $\mathcal T\in \mathbb R^{I_1\times \cdots \times I_N}$ and a real matrix $\jz U\in \mathbb R^{J\times I_k}$, the mode-$k$ product of $\zl T$ and $\jz U$ is given by $\mathcal Y = \zl T\times_k \jz U$, with entries 
    \[y_{i_1\cdots i_{k-1}ji_{k+1}\cdots i_N}=\sum_{i_k=1}^{I_k}\nolimits t_{i_1i_2\cdots i_N}u_{ji_k} = \sum_{i_k=1}^{I_k}\nolimits  u_{ji_k} t_{i_1i_2\cdots i_N}.\]
When $\zl T$ and $\jz U$ are both quaternionic, due to the non-commutativity, in general 
\[\sum_{i_k=1}^{I_k}\nolimits t_{i_1i_2\dots i_N}u_{ji_k} \neq  \sum_{i_k=1}^{I_k}\nolimits u_{ji_k}t_{i_1i_2\cdots i_N}. \]
To resolve this issue for quaternion-matrix product, it is possible to define left and right mode-$k$ products respectively, depending on  the multiplication of the matrix from left or right. The concept of left and right quaternion tensor-matrix products was first introduced by Imhogiemhe et al. \cite{osimone2023low}   recently; however, for our purpose we need a somewhat different definition, and we will point out the differences later.
    \begin{definition}[Left and right mode-$k$ products]    \label{lproduct}
    The  left mode-$k$ product of   $\zl{T} \in \quat{I_1\times I_2\times \cdots \times I_N}$ with     $\jz{U}\in \quat{J\times I_k}$ is denoted with the notation $\times_k^L$ as
    \[\zl{Y}=\jz{U}\times_k^L\zl{T} \in \quat{I_1\times \cdots \times I_{k-1}\times J\times I_{k+1}\times \cdots \times I_N}\]
    with entries
    \[y_{i_1\cdots i_{k-1}ji_{k+1}\cdots i_N}=\sum_{i_k=1}^{I_k}\nolimits u_{ji_k}t_{i_1i_2\cdots i_N}.\]
    The right  mode-$k$ product of  $\zl{T}  $ with   $\jz{V}\in \quat{I_k\times J}$ is denoted with the notation $\times_k^R$ as
    \[\zl{Y}=\zl{T}\times_k^R \jz{V} \in \quat{I_1\times \cdots \times I_{k-1}\times J\times I_{k+1}\times \cdots \times I_N}\]
    with entries
    \begin{equation}\label{eq:def:lr_prod:1}y_{i_1\cdots i_{k-1}ji_{k+1}\cdots i_N}=\sum_{i_k=1}^{I_k}\nolimits t_{i_1i_2\cdots i_N}v_{i_kj}.\end{equation}
\end{definition}

The left mode-$k$ product,     defined in \cite{miao2023quat} without the prefix ``left'' (\cite{miao2023quat}   used the conventional notation $\times_k$), coincides with its real counterpart.  The right product is slightly different from its real counterpart: note that
\eqref{eq:def:lr_prod:1} sums the entries of $t_{i_1i_2\cdots i_N}v_{i_kj}$ by running over all the row indices of $v_{i_k j}$ instead of its column indices\footnote{We could have defined the right product by the conventional summation in the real setting  $y_{i_1\cdots j \cdots i_N}=\sum_{i_k=1}^{I_k}t_{i_1i_2\cdots i_N}v_{ji_k}$. If this is the case, then successively multiplying two matrices of proper  sizes on the same mode  would  yield  $\zl T\times_k^R\jz U\times_k^R \jz V = \zl T\times_k^R \jz W$, where $\jz W$ has entries $w_{ji} = \sum_k u_{ki}v_{jk} $, i.e., $\jz{W}$ cannot be computed by the usual matrix products. In fact, \cite{schulz2014using,osimone2023low} defined $\jz W$ as the right quaternion product   $\jz W = \jz V\cdot_R\jz U$, which does not share the ordinary algebraic properties of the standard product  $\jz U\jz V$.  By contrast,   adopting definition \eqref{eq:def:lr_prod:1} yields $\zl Y = \zl T\times_k^R (\jz U\jz V)$ (see Proposition \ref{left}),   which is more convenient in the computation and analysis.    
}.

The left and right products are consistent with the matrix products, i.e.,
given $\jz U\in\quat{I\times J }, \jz S\in \quat{J\times K}, \jz V\in \quat{K\times L}$, one has 
$ 
\jz U\times_1^L \jz S\times_2^R\jz V = \jz U\jz S\jz V.
$

The symbols $\times_k^L$ and $\times_k^R$ are taken from  \cite{osimone2023low}.  We point out the differences between Definition \ref{lproduct} and those in \cite{osimone2023low}. \cite{osimone2023low} defined three types of left and right mode-$k$ products:  
\begin{itemize}
    \item left mode-$1$ product: $\bigxiaokuohao{\zl T \times_1^L \jz U  }_{j i_2\cdots i_N} = \sum^{I_1}_{i_1=1}u_{j i_1}t_{i_1 i_2\cdots i_N}$ for $\jz U\in \quat{J\times I_1}$;
    \item right mode-$N$ product: $\bigxiaokuohao{ \zl{T}\!\times_N^R \jz U }_{i_1\cdots i_{N-1}j }\!\! = \!\sum^{I_N}_{i_N=1} t_{i_1\cdots i_N} u_{j i_N}  $ for $\jz U\in \quat{J\times I_N}$;
    \item mode-$d$ ($2\leq d\leq N-1$) product of a real matrix: $\bigxiaokuohao{ \zl T\times_d \jz U }_{i_1\cdots j\cdots i_N} = \sum^{I_d}_{i_d=1}t_{i_1\cdots i_d\cdots i_N}u_{j i_d} $ for $U\in \mathbb R^{J\times I_d}$.
\end{itemize}
The differences are:
\begin{itemize}
    \item the left and right products of \cite{osimone2023low} are    respectively defined for mode-$1$ and   mode-$N$, while ours are defined for any mode. The right product is different: as explained in the footnote, we adopt definition \eqref{eq:def:lr_prod:1}. 
    \item for mode-$d$ $(2\leq d\leq N-1)$ products, \cite{osimone2023low} requires a real matrix $\jz U$, while ours can be quaternion.
\end{itemize}

The following basic properties of left and right products can be directly verified; the properties of left products have been summarized in \cite{miao2023quat}.
\begin{proposition} For quaternion tensor $\zl T$ and quaternion matrix $\jz U$ of proper sizes:
    \begin{enumerate}
        \item Relations between left/right mode-$k$ products and their unfoldings: 
             \[\zl{Y}=\jz{U}\times_k^L\zl{T}\Leftrightarrow \jz{Y}_{[k]}^L=\jz{U}\jz{T}_{[k]}^L~{\rm and}~\zl{Y}=\zl{T}\times_k^R\jz{V}\Leftrightarrow \jz{Y}^R_{[k]}=\jz{T}_{[k]}^R\jz{V}.\]
        \item Mode-$k$ product  of matrices of proper sizes on the same mode:  
              \[ \jz{U}_2\times_k^L\jz{U}_1 \times_k^L\zl{T} =(\jz{U}_2\jz{U}_1)\times_k^L\zl{T}~{\rm and}~\zl{T}\times_k^R \jz{V}_1\times_k^R \jz{V}_2= \zl{T}\times_k^R (\jz{V}_1\jz{V}_2).\]
        \item Mode-$k$ products of matrices on different modes: if $k_1\neq k_2$, in general they do not commute:
              \[\jz{U}_{k_2}\times_{k_2}^L\jz{U}_{k_1}\times_{k_1}^L\zl{T} \neq \jz{U}_{k_1} \times_{k_1}^L\jz{U}_{k_2}\times_{k_2}^L\zl{T}~{\rm and}~\zl{T}\times_{k_1}^R \jz{V}_{k_1}\times_{k_2}^R \jz{V}_{k_2}\neq \zl{T}\times_{k_2}^R \jz{V}_{k_2}\times_{k_1}^R \jz{V}_{k_1}. \]
    \end{enumerate}
    \label{left}
\end{proposition}

We end this section by presenting a lemma concerning unfolding of mode-$k$ products    by means of the   Kronecker product, which will be used in analyzing properties of quaternion HOSVD. Its proof is in \refsect{sec:proof_lem_k_mode_right_unfolding}. 

\color{black}
\begin{lemma}\label{lem:mode-k_unfolding}
  Let $\zl{T}\in \mathbb{H}^{I_1\times I_2\times \cdots \times I_N}$ be an $N$-th order quaternion tensor. Let $\jz{U}_i\in \mathbb{H}^{J_i\times I_i}$, $i = 1,\ldots,k$ be quaternion matrices with \color{black} $k<N$,\color{black} and $\jz{I}^{(t)} \in \mathbb{R}^{I_t \times I_t}$ denote the identity matrix   for mode-$t$.  If $\zl{Y}= \jz{U}_k\times^L_k \jz{U}_{k-1}\times^L_{k-1}  \cdots  \jz{U}_1\times^L_1 \zl{T}$. Then for any $N\geq j>k$, the right mode-$j$ unfolding of $\zl{Y}$ is given by
    \begin{equation}
        \jz{Y}^L_{[j]}=\bigxiaokuohao{(\jz{I}^{(N)} \otimes \cdots \otimes \jz{I}^{(j+1)} \otimes \jz{I}^{(j-1)} \otimes \cdots \otimes  \jz{I}^{(k+1)} \otimes \jz{U}_{k} \otimes \cdots \otimes \jz{U}_{1})(\jz{T}^L_{[j]})^T}^T;
        \label{equ:right_j>k}
    \end{equation}
in particular    when $j=k$, we have
    \begin{equation}\label{equ:left_mode_k_product}
        \jz{Y}^L_{[j]}=\jz{U}_j \bigxiaokuohao{(\jz{I}^{(N)} \otimes \cdots \otimes \jz{I}^{(j+1)} \otimes \jz{I}^{(j-1)} \otimes \cdots \otimes  \jz{I}^{(k+1)} \otimes \jz{U}_{k} \otimes \cdots \otimes \jz{U}_{1})(\jz{T}^L_{[j]})^T}^T.
    \end{equation}
    Similarly, if $\zl{Y}= \zl{T} \times^R_N \jz{V}_N \cdots \times^R_k \jz{V}_k$, where $k> 1$, and $\mathbf{V}_i \in \mathbb{H}^{I_i \times J_i}$ for $i = k, \ldots, N$. Then for any $k>j\geq 1$ we have
    \begin{equation*}
        \jz{Y}^R_{[j]}=\bigxiaokuohao{(\jz{T}^R_{[j]})^T(\jz{V}_N \otimes \cdots \otimes \jz{V}_{k} \otimes \jz{I}^{(k-1)} \otimes \cdots \otimes  \jz{I}^{(j+1)} \otimes \jz{I}^{(j-1)} \otimes \cdots \otimes \jz{I}^{(1)})}^T;
    \end{equation*}
   when $j=k$ we have
    \begin{equation}\label{equ:right_mode_k_product}
        \jz{Y}^R_{[j]}=\bigxiaokuohao{(\jz{T}^R_{[j]})^T(\jz{V}_N \otimes \cdots \otimes \jz{V}_{k} \otimes \jz{I}^{(k-1)} \otimes \cdots \otimes  \jz{I}^{(j+1)} \otimes \jz{I}^{(j-1)} \otimes \cdots \otimes \jz{I}^{(1)})}^T\jz{V}_j.
    \end{equation}
\end{lemma}

When the tensor and matrices are real, \eqref{equ:left_mode_k_product} reduces to
\color{black}
\[\jz{Y}^L_{[j]}=\jz{U}_j \jz{T}^L_{[j]}(\jz{I}^{(N)} \otimes \cdots \otimes \jz{I}^{(j+1)} \otimes \jz{U}_{j-1} \otimes \cdots \otimes \jz{U}_{1})^T,\]
coinciding with its real counterpart  \cite{kolda2009tensor}. Similarly, \eqref{equ:right_mode_k_product} reduces to
 $\jz{Y}^R_{[j]}=(\jz{V}_N \otimes \cdots \otimes \jz{V}_{j+1} \otimes \jz{I}^{(j-1)} \otimes \cdots \otimes \jz{I}^{(1)})^T\jz{T}^R_{[j]}{\jz V}_j  $.

\color{black}

\section{A Parallelizable Quaternion HOSVD}\label{decomposition}
\color{black}
This section  first shortly reviews   HOSVD, sequential HOSVD, and     the quaternion HOSVD proposed in \cite{miao2023quat}  (L-QHOSVD). Then            the parallelizable quaternion HOSVD is introduced, with     ordering and all-orthogonality properties being studied.  Finally,  an essential difference between L-QHOSVD and the proposed QHOSVD in the matrix case is pointed out.  
\subsection{HOSVD}
HOSVD was first introduced by De Lathauwer et al. \cite{hosvd}, along with theoretical guarantees that the core tensor satisfies both the ordering and all-orthogonality properties for its subtensors.     
\begin{theorem}[HOSVD \cite{hosvd}]
    Every complex tensor $\zl{T}\in \mathbb{C}^{I_1\times I_2\times \cdots \times I_N}$ can be written as  
     ${\zl T} = {\zl S}\times_1 {\jz U}_1\times_2 {\jz U}_2 \cdots \times_N {\jz U}_N   $,
    with ${\jz U}_n\in \mathbb{C}^{I_n\times I_n}(n = 1,\dots ,N)$ being  unitary   and ${\zl S}\in \mathbb{C}^{I_1\times I_2\times \cdots \times I_N}$. The subtensors ${\zl S}_{i_n} = \alpha$, obtained by fixing the $n$-th index to $\alpha$, have properties of 
    \begin{enumerate}
        \item ordering for all possible values of n:  $\|{\zl S}_{i_n = 1}\|_F \geq \|{\zl S}_{i_n = 2}\|_F \geq \cdots \geq \|{\zl S}_{i_n = I_n}\|_F   $;
        \item all-orthogonality:  
         ${\zl S}_{i_n = \alpha}$ and ${\zl S}_{i_n = \beta}$ are orthogonal for all possible values of $n$, $\alpha$ and $\beta$ subject to $\alpha \neq \beta$:\[\langle {\zl S}_{i_n = \alpha},{\zl S}_{i_n = \alpha}\rangle = 0\quad {\rm when}\quad \alpha \neq \beta. \]
    \end{enumerate} 
\end{theorem}

  (Truncated) HOSVD is given in Algorithm \ref{algo:hosvd}. When $I_k^\prime = I_k,k = 1,\ldots,N$, it is the exact decomposition.   Vannieuwenhoven et al. \cite{vannieuwenhoven2012new} proposed a sequential variant of HOSVD (Algorithm \ref{algo:sequential_hosvd}), aiming at accelerating the computation of the truncated HOSVD. The reconstruction error of truncated HOSVD can be bounded by its tail energy. Specifically, let $\hat{\mathcal S},\hat{\jz{U}_k}$ be generated By Algorithm \ref{algo:hosvd}. Then \cite[Property 10]{hosvd}:
 \begin{align}
    \label{eq:error_hosvd_tail_energy}
\|\mathcal T - \hat{\zl{S}}\times_1 \hat{\jz{U}}_1\times_2 \hat{\jz{U}}_2 \cdots \times_N \hat{\jz{U}}_N\|_F^2 \leq \sum_{i_1 = I_1^\prime + 1}^{I_1}\nolimits \sigma_{i_1}^2(\jz{T}_{[1]} ) + \cdots +  \sum_{i_N = I_N^\prime + 1}^{I_N}\nolimits \sigma_{i_N}^2(\jz{T}_{[N]}),
 \end{align}
 where $\sigma_{i}(\cdot)$ denotes the $i$-th singular value of the   matrix. Sequential HOSVD has a similar error bound.
 \begin{small}
 \begin{algorithm} 
    \small
    \caption{HOSVD \cite{hosvd}}
    \begin{algorithmic}[1]
        \REQUIRE{ $\zl{T}\in \mathbb{C}^{I_1\times I_2\times\cdots \times I_N}$, truncated rank $(I_1',I_2',\dots ,I_N')$.}
        \ENSURE{$\hat{\zl S}$, $\{\hat{\jz U}_1, \hat{\jz U}_2, \cdots ,\hat{\jz U}_N\}.$}
    \FOR{$k = 1:N$}
        \STATE{$\hat{\jz U}_k \in\mathbb C^{I_k\times I_k^\prime}\leftarrow $ left leading $I_k'$ singular vectors of ${\jz T}_{[k]}$~~~~~~{\color{lightgray} \% ${\jz T}_{[k]}$ is the mode-$k$ unfolding of $\zl T$}}
    \ENDFOR
    \STATE{$\hat{\zl S} \in\mathbb C^{I_1^\prime\times \cdots \times I_N^\prime} \leftarrow {\zl T}\times_1 \hat{\jz U}_1^H\times_2 \hat{\jz U}_2^H\cdots \times_N \hat{\jz U}_N^H$}
    \end{algorithmic}
    \label{algo:hosvd}
 \end{algorithm}
\end{small}
\begin{small}
 \begin{algorithm}
    \small
    \caption{Sequential HOSVD \cite{vannieuwenhoven2012new}}
    \begin{algorithmic}[1]
        \REQUIRE{ $\zl{T}\in \mathbb{C}^{I_1\times I_2\times\cdots \times I_N}$, truncated rank $(I_1',I_2',\dots ,I_N')$.}
        \ENSURE{$\hat{\zl T}$, $\{\hat{\jz U}_1, \hat{\jz U}_2, \cdots ,\hat{\jz U}_N\}.$}
    \STATE{$\hat{\zl T} \leftarrow {\zl T}$}
    \FOR{$k = 1:N$}
        \STATE{$\hat{\jz U}_k\in\mathbb C^{I_k\times I_k^\prime} \leftarrow $   left leading $I_k'$ singular vectors of $\hat{\jz T}_{[k]}$~~~~~~{\color{lightgray} \% $\hat{\jz T}_{[k]}$ is the mode-$k$ unfolding of $\hat{\zl T}$}}
        \STATE{$\hat{\zl T} \in\mathbb C^{I_1^\prime\times \cdots\times I_k^\prime \times I_{k+1}\times \cdots \times I_N}\leftarrow \hat{\zl T}\times_k \hat{\jz U}_k^H$}
    \ENDFOR
    \end{algorithmic}
    \label{algo:sequential_hosvd}
 \end{algorithm}
\end{small}

\subsection{L-QHOSVD} \label{sec:lqhosvd}
\cite{miao2023quat} introduces quaternion HOSVD (L-QHOSVD for short) that   is in the spirit  of    sequential HOSVD. Specifically, given $\zl T\in\quat{I_1\times\cdots\times I_N}$,  for   $k=N:-1:1$, L-QHOSVD computes the factor matrix $\jz U_k\in\quat{I_k\times I_k}$ from the left singular vectors of   the \emph{left} mode-$k$ unfolding  of $\zl T_{k+1}$: $ (\jz T_{k+1} )_{[k]}^L$, where the tensors $\zl T_k$'s are defined recursively as $\zl T_{k}:= \jz U_{k}^H \times_{k}^L \zl T_{k+1}$ with $\zl T_{N+1}:=\zl T$. The core tensor is given by $\zl S = \zl T_1$.  The following properties hold.  
\begin{theorem}[L-QHOSVD \cite{miao2023quat}]
      Let $\zl{T}\in \quat{I_1\times I_2\times\cdots \times I_N}$;   let the unitary matrices $\jz{U}_k\in\quat{ I_k\times I_k}$ ($k=1,\ldots,N$) and the core tensor $\zl S=\jz U_1^H\times_1^L \zl T_2=\cdots = \jz U_1^H\times_1^L\cdots U_N^H\times_N^L \zl T\in\quat{I_1\times I_2\times\cdots \times I_N}$ be generated as above. 
    Then, $\zl T$ can be decomposed as
        $\zl{T}=\jz{U}_N \times_N^L\cdots  \jz{U}_2 \times_2^L \jz{U}_1\times_1^L\zl{S}$. 
      $\zl{S}\in \quat{I_1\times I_2\times \cdots \times I_N}$  admits the following properties:
    \begin{enumerate}[i)]
        \item Ordering:
        \[\|\zl{S}_{i_k=1}\|_F=(\sigma_1^L)^{(k)}\geq \|\zl{S}_{i_k=2}\|_F=(\sigma_2^L)^{(k)} \geq \cdots \geq \|\zl{S}_{i_k=I_k}\|_F=(\sigma_{I_k}^L)^{(k)}\geq 0\]
        for any $k=1,2,\dots ,N$,  where $\zl S_{i_k = \alpha}$ is the $(N-1)$-th order tensor obtained by fixing the $k$-th index of $\zl S$ to be $\alpha$, 
        and $(\sigma_i^L)^{(k)}$'s are the singular values of 
        $(\jz T_{k+1})^L_{[k]}$, arranged in a  descending order.
        \item Orthogonality: The left orthogonality is satisfied   in mode-$1$: 
         $\langle \zl{S}_{i_1=\alpha} , \zl{S}_{i_1=\beta} \rangle_L =0 \quad \text{when} \quad \alpha\neq \beta. $
    \end{enumerate}
    \label{L}
\end{theorem}

\begin{remark} \label{rmk:lqhosvd}
\cite{miao2023quat} did not consider the rank truncated L-QHOSVD as that in Algorithm \ref{algo:sequential_hosvd}. Instead,  it shrinks the core tensor \cite[Sect. 5.2]{miao2023quat}. Specifically,   once $\zl S\in \quat{I_1\times I_2\times\cdots \times I_N},\jz U_1\in\quat{I_1\times I_1},\ldots,\jz U_N \in\quat{I_N\times I_N}$ are obtained from the full decomposition of L-QHOSVD, then zero  entries of $\zl{S}$ whose norms exceed a given threshold to form $\hat{\zl{S}}$, 
 and finally reconstruct    $\hat{\zl T} = \jz U_N\times_N^L \cdots \jz U_1\times_1^L\hat{\zl S}$. 
\end{remark}

\subsection{The proposed QHOSVD}  Modern parallel computing architechture can accelerate computations. However, a straightforward quaternion-domain extension of Algorithm \ref{algo:hosvd} (HOSVD), which computes each factor matrix in parallel, runs into issues. This is illustrated via an   example.   
\begin{example}  \label{example:1} Consider   ${\zl T}\in \mathbb{H}^{3\times 3\times 3\times 3}$ with   entries given as $\zl{T}(a,b,c,d) = \frac{2}{(a-3) + b\mathbf{i} + c\mathbf{j} + d\mathbf{k}}$, where $a,b,c,d\in \{1,2,3\}$. Set truncated rank $(I_1^\prime,\!\ldots\!,I_4^\prime) = (2,2,2,2)$.   Denote  $\Sigma_k^L$ as the diagonal singular value matrix of the left mode-$k$ unfolding $\jz{T}^L_{[k]}$, $k=1,2,3,4$,    which are respectively listed as follows:
\begin{align*}
    \Sigma^L_1 = {\rm diag}(5.0109,0.4830,0.0910),~\Sigma^L_2 = {\rm diag}(5.0134,0.4621,0.0605),\\
    \Sigma^L_3 = {\rm diag}(5.0134,0.4621,0.0605),~\Sigma^L_4 = {\rm diag}(5.0134,0.4621,0.0605).
\end{align*}
Let $\hat{\zl T} = \hat{\jz{U}}_4\times^L_4 \cdots\hat{\jz{U}}_1\times^L_1\hat{\zl{S}} \in\quat{3\times 3\times 3\times 3}$ be the reconstructed tensor with $\hat{\zl{S}}~ {  and} ~\hat{\jz{U}_k}$'s being computed by a direct quaternion adaption of Algorithm \ref{algo:hosvd}. 
One can reproduce  the reconstructed error $\|\zl{T}-\hat{\zl T}\|_F^2=0.1166$. On the other hand,  the tail energy is $\sigma_3(\Sigma^L_1)^2+\sigma_3(\Sigma^L_2)^2+\sigma_3(\Sigma^L_3)^2+\sigma_3(\Sigma^L_4)^2=0.01926$. 
\end{example}

In fact, examples like this are not rare: 
  unlike the real truncated HOSVD bound \eqref{eq:error_hosvd_tail_energy}, the reconstruction error can far exceed the tail energy, i.e., a naive quaternion extension of Algorithm \ref{algo:hosvd} yields unstable truncation.

In view of the above limitation,  we consider an exact decomposition  of the following fashion in this section (its truncated version will be studied in the next section):
\[
\zl T= \jz{U}_m\times_m^L\cdots \jz{U}_{1}\times_{1}^L\zl{S} \times_N^R \jz{V}_N^H\cdots \times_{m+1}^R\jz{V}^H_{m+1},
\]
where $\zl{S}$ is the core tensor and $\jz{U}_k$'s and $\jz{V}_k$'s are factor matrices. 
To formally     present the idea, 
 we without loss of generality assume the order of the  tensor is $N=2m$, while  odd order tensors can be treated as even order tensors:  For instance,  a third-order tensor $\zl T \in \mathbb{H}^{m\times n\times p}$ can also be viewed as a four order tensor  $\zl T \in \mathbb{H}^{m\times n\times p\times 1}$.  

\begin{figure}[htbp]
    \centering
    \resizebox{\textwidth}{!}{
    \begin{tikzpicture}[
    node distance=2cm and 1.8cm,
    every node/.style={draw, minimum size=8mm, font=\sffamily},
    ->, >=Stealth
]

\node (A) {$\zl{T}\in \quat{I_1\times I_2\times\cdots \times I_N}$};
\node (B1) [right=1cm of A, yshift=1cm] {\quad compute ${\jz U}_m$ and ${\zl L}_m$\quad ~};
\node (Bn) [right=1cm of B1] {~~compute ${\jz U}_1$ and ${\zl L}_1$~~};
\node (C1) [right=1cm of A, yshift=-1cm] {compute ${\jz V}_{m+1}$ and ${\zl R}_{m+1}$};
\node (Cn) [right=1cm of C1] {compute ${\jz V}_{N}$ and ${\zl R}_{N}$};
\node (D)  [right=10cm of A] {Synthesize core tensor ${\zl S}$};
\node (E)  [right = 1cm of D] {$\zl{S}$, $\{\jz{U}_1,\cdots ,\jz{U}_m,\jz{V}_{m+1},\cdots ,\jz{V}_{N}\}.$};

\node at ($(A)+(0,1cm)$) {Input};
\node at ($(B1)+(2.5,1)$) {Processor 1};
\node at ($(C1)+(2.5,-1)$) {Processor 2};
\node at ($(E)+(0,1)$) {Output};

\draw[line width=1.2pt] (A) -- (B1);
\draw[line width=1.2pt] (A) -- (C1);
\draw[dotted, line width = 1.5pt] (B1) -- (Bn);
\draw[dotted, line width = 1.5pt] (C1) -- (Cn);
\draw[line width = 1.2pt] (Bn) -- (D);
\draw[line width = 1.2pt] (Cn) -- (D);
\draw[line width = 1.2pt] (D) -- (E);
\end{tikzpicture}
    }
    \captionsetup{aboveskip=-1pt, belowskip=-4pt}
    \caption{Parallel process of the proposed QHOSVD}     \label{fig:parallel_process}
\end{figure}

The decomposition procedure is illustrated  in Fig. \ref{fig:parallel_process}. The detail is    as follows: For  $\zl T \in \quat{I_1 \times \cdots \times I_N}$,   initialize $\zl L_{m+1} = \zl T = \zl R_m$, and simultaneously launch two independent recursive procedures on two   processors:
\begin{itemize}
    \item Processor 1: sequentially computes left singular vectors of $(\jz{L}_{i+1})^L_{[i]}$ to form left core tensors $\zl L_i$ for $i = m, \dots, 1$ in the spirit of sequential HOSVD.    $(\jz{L}_{i+1})^L_{[i]}$ is the left mode-$i$ unfolding of $\zl L_{i+1}$. 
    \item Processor 2: sequentially computes right singular vectors of $(\jz{R}_{j-1})^R_{[j]}$ to form   right core tensors $\zl R_j$ for $j = m+1, \dots, N$ in the spirit of sequential HOSVD.  $(\jz{R}_{j-1})^R_{[j]}$ is the right mode-$j$ unfolding of $\zl R_{j-1}$.
\end{itemize}
The recursion steps on each processor are presented below:
\begin{equation}
    \begin{split}
        &\qquad ({\rm Processor~1}) \qquad\qquad\qquad~~~ ({\rm Processor~2}) \\
        &\zl{L}_{m}  :=\jz{U}_m^H \times_m^L \zl{L}_{m+1} \qquad\quad~~ \zl{R}_{m+1}  :=\zl{R}_m \times_{m+1}^R \jz{V}_{m+1} \\
        &\zl{L}_{m-1}  :=\jz{U}_{m-1}^H \times_{m-1}^L \zl{L}_{m}   \quad~~\zl{R}_{m+2}  :=\zl{R}_{m+1} \times_{m+2}^R \jz{V}_{m+2}\\
         &\qquad \vdots  \qquad\qquad\qquad\qquad\qquad\qquad \vdots  \\
        &\zl{L}_1  :=\jz{U}_1^H\times_1^L \zl{L}_2   \qquad\qquad~~~~ \zl{R}_{N} :=\zl{R}_{N-1} \times_{N}^R \jz{V}_{N},
    \end{split}
    \label{eq:parallel_recursion}
\end{equation}
in which $\jz{U}_i\in \mathbb{H}^{I_i\times I_i}$ and $\jz{V}_j\in \mathbb{H}^{I_j\times I_j}$ are obtained from the SVD  of the left mode-$i$ unfoldings of $\zl L_{i+1}$ and right mode-$j$ unfoldings of $\zl R_{j-1}$, respectively, as follows:
\begin{equation}
\begin{split}
(\jz{L}_{i+1})^L_{[i]} = \jz{U}_i \Sigma^L_i \jz{V}_i^H\in \mathbb{H}^{I_i\times (\prod_{\substack{s = 1\\s\neq i}}^{N} I_s)}, \\
(\jz{R}_{j-1})^R_{[j]} = \jz{U}_j \Sigma^R_j \jz{V}_j^H\in \mathbb{H}^{(\prod_{\substack{s = 1\\s\neq j}}^{N} I_s)\times I_j},
\end{split}
\label{eq:svd_of_left_and_right_tensor}
\end{equation}
where $\Sigma^L_i$ and $\Sigma^R_j$ are the diagonal  singular value matrices, whose diagonal entries are arranged in descending order:
$$
(\sigma_1^L)^{(i)}, \dots, (\sigma_{I_i}^L)^{(i)}, \quad \text{and} \quad (\sigma_1^R)^{(j)}, \dots, (\sigma_{I_j}^R)^{(j)}.
$$
The final core tensor $\zl S$ is constructed by combining the two halves:
\begin{align} \label{eq:ts_qhosvd_core_tensor}
\zl{S} =& {\jz U}_1^H\times^L_1 \cdots {\jz U}_m^H\times^L_m {\zl R}_N = \zl{L}_1 \times^R_{m+1} \jz{V}_{m+1} \cdots \times^R_N \jz{V}_N \\
=& \jz{U}_1^H \times^L_1 \cdots \times^L_m \jz{U}_m^H \times \zl{T} \times^R_{m+1} \jz{V}_{m+1} \cdots \times^R_N \jz{V}_N.
\end{align}
The corresponding computational steps are detailed in Algorithm \ref{algo:ts}, denoted as TS-QHOSVD.

\color{black}
Based on the recursive process described in \eqref{eq:parallel_recursion}, the intermediate tensors $\zl L_i$ and $\zl R_j$ admit explicit representations through the original tensor $\zl T$ and the factor matrices $\jz{U}_i$, $\jz{V}_j$ as follows:
\begin{align}
    \zl{L}_{i} = \jz{U}_{i}^H \times^L_{i} \cdots \times^L_m \jz{U}_m^H \times^L_m \zl{T}, \quad
    \zl{R}_{j} = \zl{T} \times_{m+1}^R \jz{V}_{m+1} \cdots \times_{j}^R \jz{V}_{j}.
    \label{eq:relation_of_Li_Rj_and_T}
\end{align}
\color{black}
Now, the properties of TS-QHOSVD are presented as follows. Without loss of generality assume $N=2m$.


\begin{small}
 \begin{algorithm}
    \small
    \caption{Two-Sided QHOSVD (TS-QHOSVD)}
    \begin{algorithmic}[1]
        \REQUIRE{ $\zl{T}\in \quat{I_1\times I_2\times\cdots \times I_N}$, denote $\zl L_{m+1} = \zl T = \zl R_{m}$.}
        \ENSURE{$\zl{S}$, $\{\jz{U}_1,\cdots ,\jz{U}_m,\jz{V}_{m+1},\cdots ,\jz{V}_{N}\}.$}
    
    \FOR{$i=m:-1:1$}   
        \STATE{$\jz{U}_{i}\leftarrow $ left singular vectors of $(\jz{L}_{i+1})^L_{[i]}$~~~~~~~~~~~~~~~{\color{gray}\% on Processor 1; $(\jz{L}_{i+1})^L_{[i]}$ the left mode-$i$ unfolding of $\zl L_{i+1}$}}
        \STATE{$\zl{L}_{i}\leftarrow \jz{U}_{i}^H \times_k^L \zl{L}_{i+1}$}
    \ENDFOR
    \FOR{$j = m+1 : N$}
        \STATE{$\jz{V}_{j}\leftarrow $ right singular vectors of $(\jz{R}_{j-1})^R_{[j]}$~~~~~~~~~~~~~{\color{gray}\% on Processor 2; $(\jz{R}_{j-1})^R_{[j]}$ the right mode-$j$ unfolding of $\zl R_{j-1}$}}
        \STATE{$\zl{R}_j\leftarrow \zl{R}_{j-1}\times_{j}^R \jz{V}_{j}$}
    \ENDFOR
    \STATE{$\zl{S} \leftarrow {\jz U}_1^H\times^L_1 \cdots {\jz U}_m^H\times^L_m {\zl R}_N$}
    \end{algorithmic}
    \label{algo:ts}
 \end{algorithm}
\end{small}

\begin{theorem}[TS-QHOSVD]
    \label{two-sided}
        Let $\zl{T}\in \quat{I_1\times I_2\times\cdots \times I_N}$. Let the unitary matrices $\jz{U}_i \in\quat{I_i\times I_i}$, $\jz{V}_{j}\in\quat{I_{j}\times I_{j}  }$   ($i=m:-1:1,j = m+1:N$), and the core tensor $\zl S$ be generated by Algorithm \ref{algo:ts}. Then
    \begin{equation}
        \zl{T}=\jz{U}_m\times_m^L\cdots \jz{U}_{1}\times_{1}^L\zl{S} \times_N^R \jz{V}_N^H\cdots \times_{m+1}^R\jz{V}^H_{m+1}.
        \label{2}
    \end{equation}
The core tensor     $\zl{S}\in \quat{I_1\times I_2\times \cdots \times I_N}$ admits   the following properties:
    \begin{enumerate}[i)]
        \item Ordering:
        \begin{align*}
            &\|\zl{S}_{i_p=1}\|_F=(\sigma_1^L)^{(p)}\geq \|\zl{S}_{i_p=2}\|_F=(\sigma_2^L)^{(p)} \geq \cdots \geq \|\zl{S}_{i_p=I_p}\|_F=(\sigma_{I_p}^L)^{(p)}\geq 0; \\
            &\|\zl{S}_{i_q=1}\|_F=(\sigma_1^R)^{(q)}\geq \|\zl{S}_{i_q=2}\|_F=(\sigma_2^R)^{(q)} \geq \cdots \geq \|\zl{S}_{i_q=I_q}\|_F=(\sigma_{I_q}^R)^{(q)}\geq 0,
        \end{align*}
        for any   $p=1,\dots ,m$ and $q=m+1,\dots ,N$, where $\zl S_{i_k = \alpha}$ is the $(N-1)$-th order subtensor obtained by fixing the $k$-th index of $\zl S$ to be $\alpha$ $(k=1,2,\dots, N)$. $(\sigma_k^L)^{(p)}$'s are the singular values of 
        $(\jz L_{p+1})^L_{[p]}$, arranged in a descending order, and
        $(\sigma_k^R)^{(q)}$'s are the singular values of  
        $(\jz R_{q-1})^R_{[q]}$, 
        arranged in a descending order.
        \item Orthogonality: In mode-$1$, the left orthogonality is satisfied;   in mode-$N$, the right orthogonality is satisfied:
        \begin{align*}
            &\langle \zl{S}_{i_1=\alpha} , \zl{S}_{i_1=\beta} \rangle_L =0
            \quad \text{when} \quad \alpha\neq \beta,~~~~
            &\langle \zl{S}_{i_N=\alpha} , \zl{S}_{i_N=\beta} \rangle_R =0 
            \quad \text{when} \quad \alpha\neq \beta.
        \end{align*}
       \item Weak orthogonality: In all modes, the weak orthogonality is satisfied, i.e., for $k=1,\ldots,N$,
        \begin{align*}
            Re\langle \zl{S}_{i_k=\alpha}, \zl{S}_{i_k=\beta} \rangle_L = Re\langle \zl{S}_{i_k=\alpha}, \zl{S}_{i_k=\beta} \rangle_R=0 \quad {\rm when} \quad \alpha\neq \beta.
        \end{align*}
    \end{enumerate}
\end{theorem}
 
\begin{remark} \label{rmk:orthogonality_ts_qhosvd}  The left, right, and weak orthogonalities was given in Definition \ref{def:orthogonality}. Comparing with   L-QHOSVD whose orthogonality is only satisfied in mode-$1$,   the orthogonality of TS-QHOSVD is satisfies in   mode-$1$ and mode-$N$. In addition, the weak orthogonality  was not discovered for L-QHOSVD.  Note that when $\zl T$ is real, the weak orthogonality property above boils down exactly to the all-orthogonality property of HOSVD.
\end{remark}


\begin{proof}[Proof of Theorem \ref{two-sided}]
First,  the representation of $\zl{T}$ \eqref{2} in terms of $\zl S$,  $\jz U_i$'s, and $\jz V_j$'s can be simply obtained from the computation of $\zl{S}$ in \eqref{eq:ts_qhosvd_core_tensor} and a recursive computation. The remaining proof is organized as follows: We first show the ordering, orthogonality, and weak orthogonality for  $\zl{R}_N$, and then establish relations between $\zl{R}_N$ and  $\zl{S}$, such that the same properties also hold for $\zl{S}$.  

\emph{Ordering:} From \eqref{eq:relation_of_Li_Rj_and_T}, for any $m+1\leq j\leq N$,   $\zl{R}_{j-1}$ can be rewritten as: 
    \begin{align*} 
        \zl{R}_{j-1}=& \zl{T} \times_{m+1}^R \jz{V}_{m+1} \cdots \times_{j-1}^R \jz{V}_{j-1}
        = \zl{R}_j\times_j^R\jz{V}_j^H = \cdots =   \zl{R}_N \times_N^R \jz{V}^H_N \dots \times_{j}^R \jz{V}^H_{j}.
    \end{align*}
    Using Lemma~\ref{lem:mode-k_unfolding} \eqref{equ:right_mode_k_product}, we can left unfolding $\zl R_{j-1}$ along   mode-$j$ to obtain
    \begin{equation}\label{eq:tensor_R_j-1_right_unfold}
        \begin{split}
            (\jz{R}_{j-1})^R_{[j]}= (\zl{R}_N \times_N^R \jz{V}^H_N \dots \times_{j+1}^R \jz{V}^H_{j+1})^R_{[j]}\jz{V}^H_j
             = \bigxiaokuohao{((\jz R_{N})^R_{[j]})^T\jz{X}}^T\jz{V}^H_j,
        \end{split}
    \end{equation}
    where $\jz{X} = ({\jz V}_N^H\otimes \cdots \otimes {\jz V}_{j+1}^H \otimes \jz{I}^{(j-1)} \otimes \cdots \otimes \jz{I}^{(1)})$ is unitary by Proposition \ref{lem:kron_unitary} and that $\jz V_k$'s are all unitary.
    
    Comparing   \eqref{eq:svd_of_left_and_right_tensor}  and \refeq{eq:tensor_R_j-1_right_unfold} and using the unitarity of $\jz{V}_j$, we have 
    \[
        {\jz U}_j\Sigma_j^R = \bigxiaokuohao{((\jz R_{N})^R_{[j]})^T\jz{X}}^T.
    \]
    Taking transpose on both sides and using the unitarity of $\jz{X}$, we obtain 
  $
    ({\jz U}_j\Sigma_j^R)^T\jz{X}^H = ((\jz R_{N})^R_{[j]})^T.
   $
    Taking transpose again on both sides, we arrive at
    \[
     (\jz R_{N})^R_{[j]} = (({\jz U}_j\Sigma_j^R)^T\jz{X}^H)^T.
    \]
    which, together with   Proposition \ref{lem:real_quater_transpose_commutat} and the fact that  $\Sigma_j^R$ is a real matrix  gives 
    \begin{equation}\label{eq:sl_k}
        \begin{split}
         (\jz R_{N})^R_{[j]} = \bigxiaokuohao{ (\Sigma_j^R)^T\jz{U}_j^T\jz{X}^H       }^T=    (({\jz U}_j)^T{\jz X}^H)^T\Sigma_j^R.
    \end{split}
    \end{equation}
    
    Although $(({\jz U}_j)^T\jz{X}^H)^T$ are no longer unitary matrices,   the Frobenius norm of each row/column   is still unity as transpose does not change the norms; this together with \eqref{eq:sl_k} shows that the $\alpha$-th row of $     (\jz R_{N})^R_{[j]}$ satisfies
    \begin{align*}
        \|(\zl{R}_N)_{i_j = \alpha}\|= \|(\jz{R}_N)_{[j]}^R(:,\alpha)\|_F = \| (({\jz U}_j)^T{\jz X}^H)^T\Sigma_j^R(\alpha,:)\|_F = (\sigma_{\alpha}^R)^{(j)},
    \end{align*}
    where the third equality follows from that     $\Sigma_j^R$ is diagonal with the $\alpha$-th diagonal entry being the   $\alpha$-th largest singular value $(\sigma_{\alpha}^R)^{(j)}$ of $(\jz{R}_{j-1})_{[j]}^R$. 
    
    Combining the above discussions, we arrive at:
    \begin{align*}
        &\|({\zl R}_N)_{i_j=1}\|_F\!=\!(\sigma_1^R)^{(j)}\geq \|({\zl R}_N)_{i_j=2}\|_F\!=\!(\sigma_2^R)^{(j)} \!\geq\! \cdots \!\geq\! \|({\zl R}_N)_{i_j=I_j}\|_F\!=\!(\sigma_{I_j}^R)^{(j)}\geq 0,
    \end{align*}
    where $j=m+1,\dots ,N$. This proves the ordering property for $\zl{R}_N$.

    \emph{Orthogonality:}  From the relation between $\zl R_{N-1}$ and $\zl R_N$ in \eqref{eq:parallel_recursion}, we have 
    \[\zl{R}_{N}=\zl{R}_{N-1} \times^R_N \jz{V}_N \Leftrightarrow  \zl{R}_{N-1}=\zl{R}_{N } \times^R_N \jz{V}_N^H . \]
    Right unfolding the above equation  along mode-$N$ and using Lemma \ref{lem:mode-k_unfolding}    yield
$        (\jz{R}_{N-1})^R_{[N]}= (\jz{R}_N)^R_{[N]} \jz{V}_N^H$ .
     This in connection with \eqref{eq:svd_of_left_and_right_tensor} (taking SVD of $(\jz{R}_{N-1})^R_{[N]}$) leads to
     \begin{align*}
          \jz{U}_N \Sigma^R_N \jz{V}^H_N = (\jz{R}_{N-1})^R_{[N]} = (\jz{R}_N)^R_{[N]} \jz{V}_N^H .
     \end{align*}
     By the unitarity of $\jz{V}_N$, we have 
   $
        (\jz{R}_N)^R_{[N]} = \jz{U}_N \Sigma^R_N.
 $
    As $ \jz{U}_N$ is also unitary, taking into account  the definition of left and right inner products in Definition \ref{def:left_right_inner_prod}, 
    \begin{align*}
        \langle (\zl{R}_N)_{i_N=\alpha} , (\zl{R}_N)_{i_N=\beta} \rangle_R &= \langle (\jz{R}_N)^R_{[N]}(:,\alpha), (\jz{R}_N)^R_{[N]}(:,\beta) \rangle_R \\ &= ((\jz{R}_N)^R_{[N]}(:,\alpha))^H (\jz{R}_N)^R_{[N]}(:,\beta) \\
        &= \Sigma^R_N(:,\alpha)^H\cdot\Sigma^R_N(:,\beta) = 0,~~~~\alpha\neq \beta.
    \end{align*}
        This proves the orthogonality property of $\zl{R}_N$ along mode-$N$.

    \emph{Weak orthogonality:}  From \eqref{eq:sl_k} and Proposition \ref{lem:conj_real_conterpart} we   obtain that   for any $k= m+1,\dots,N$,
    \begin{small}
    \begin{align*}\small
        Re\langle (\zl{R}_N)_{i_k=\alpha}, (\zl{R}_N)_{i_k=\beta} \rangle_R &= Re\bigxiaokuohao{((\jz{R}_N)^R_{[k]}(\alpha,:))^H(\jz{R}_N)^R_{[k]}(\beta,:)} \\
        & = Re\bigxiaokuohao{((({\jz U}_k)^T\jz{X})^T\Sigma_k^R(\alpha,:))^H\cdot(({\jz U}_k)^T\jz{X})^T\Sigma_k^R(\beta,:)} \\
        & = Re\bigxiaokuohao{(\Sigma_k^R(\alpha,:))^T\overline{(({\jz U}_k)^T\jz{X})}\cdot(({\jz U}_k)^T\jz{X})^T\Sigma_k^R(\beta,:)} \\
        & = (\Sigma_k^R(\alpha,:))^TRe\bigxiaokuohao{\overline{(({\jz U}_k)^T\jz{X})}\cdot(({\jz U}_k)^T\jz{X})^T}\Sigma_k^R(\beta,:) \\
        & \overset{{\rm Prop. }~\ref{lem:conj_real_conterpart}}{=} (\Sigma_k^R(\alpha,:))^TRe\bigxiaokuohao{({\jz U}_k)^T\jz{X}\cdot(({\jz U}_k)^T\jz{X})^H}\Sigma_k^R(\beta,:) \\
        & \overset{\jz{X}~{\rm unitary}  }{=}  (\Sigma_k^R(\alpha,:))^TRe\bigxiaokuohao{{\jz U}_k^T\cdot\overline{\jz U}_k}\Sigma_k^R(\beta,:) \\
        & \overset{{\rm Prop. }~\ref{lem:conj_real_conterpart}}{=} (\Sigma_k^R(\alpha,:))^TRe(({\jz U}_k)^H{\jz U}_k))\Sigma_k^R(\beta,:) \\
        & = (\Sigma_k^R(\alpha,:))^T\Sigma_k^R(\beta,:))  = 0,~~~~\alpha\neq \beta,
    \end{align*}
    \end{small}
   where the fourth equality comes from $Re(ab) = aRe(b)$ for real $a$ and quaternion $b$. This    proves  the weak-orthogonality of subtensors of $\zl R_N$ along modes $m+1,\ldots,N$.

  The remaining task is to establish the relation between $\zl{R}_N$ and $\zl{S}$ such that the proved properties can be transferred to $\zl{S}$. 
    From \eqref{eq:ts_qhosvd_core_tensor}, we have 
    \begin{align*} 
        {\jz U}_m\times^L_m \cdots {\jz U}_1\times^L_1 \zl S= {\zl R}_N.
    \end{align*}
   Using Lemma \ref{lem:mode-k_unfolding} \eqref{equ:right_j>k}, unfolding ${\zl R}_N$ along mode-$j$ $(j>m)$ yields:
    \begin{align*}
        (\jz R_N)^R_{[j]} = ({\jz I}^{(N)}\otimes \cdots \otimes {\jz I}^{(j+1)} \otimes {\jz I}^{(j-1)} \otimes \cdots \otimes {\jz I}^{(m+1)} \otimes \jz{U}_m \otimes \cdots {\jz U}_1){\jz S}^R_{[j]}.
    \end{align*}
  Set ${\jz Z} = ({\jz I}^{(N)}\otimes \cdots \otimes {\jz I}^{(j+1)} \otimes {\jz I}^{(j-1)} \otimes \cdots \otimes {\jz I}^{(m+1)} \otimes \jz{U}_m \otimes \cdots {\jz U}_1)$, which is unitary  by Proposition \ref{lem:kron_unitary}. This together with the obtained properties for $\zl{R}_N$  implies the following characteristics for the core tensor $\zl{S}$:

    (1) Ordering: for  $j=m+1,\ldots,N$,
    \begin{align*}
        \|{\zl S}_{i_j = \alpha}\|_F = \|{\jz S}^R_{[j]}(:,\alpha)\|_F = \|\jz{Z}^H(\jz R_N)^R_{[j]}(:,\alpha)\|_F = \|(\jz R_N)^R_{[j]}(:,\alpha)\|_F = (\sigma_\alpha^R)^{(j)},
    \end{align*}
    and so
 $
        \|{\zl S}_{i_j=1}\|_F\!=\!(\sigma_1^R)^{(j)}\geq \| {\zl S}_{i_j=2}\|_F\!=\!(\sigma_2^R)^{(j)} \!\geq\! \cdots \!\geq\! \|{\zl S}_{i_j=I_j}\|_F\!=\!(\sigma_{I_j}^R)^{(j)}\geq 0 
 $.
    
    (2) Orthogonality:
    \begin{align*}
        \langle (\zl{S})_{i_N=\alpha} , (\zl{S})_{i_N=\beta} \rangle_R &= \langle (\jz{S})^R_{[N]}(:,\alpha), (\jz{S})^R_{[N]}(:,\beta) \rangle_R  = ((\jz{S}_N)^R_{[N]}(:,\alpha))^H (\jz{S}_N)^R_{[N]}(:,\beta) \\
        &= (\jz{Z}^H(\jz{R}_N)^R_{[N]}(:,\alpha))^H (\jz{Z}^H\jz{R}_N)^R_{[N]}(:,\beta) \\
        &= ((\jz{R}_N)^R_{[N]}(:,\alpha))^H (\jz{R}_N)^R_{[N]}(:,\beta) = 0,~~~~\alpha\neq \beta;
    \end{align*}

    (3) Weak-orthogonality: for $k=m+1,\ldots,N$, 
    \begin{align*}
        Re\langle (\zl{S})_{i_k=\alpha}, (\zl{S})_{i_k=\beta} \rangle_R  
        &= Re\bigxiaokuohao{\jz{S}^R_{[k]}(\alpha,:)^H\jz{S}^R_{[k]}(\beta,:)} =Re\bigxiaokuohao{(\jz Z^H(\jz{R}_N)^R_{[k]}(\alpha,:))^H\jz Z^H(\jz{R}_N)^R_{[k]}(\beta,:)} \\
        &= Re\bigxiaokuohao{((\jz{R}_N)^R_{[k]}(\alpha,:))^H(\jz{R}_N)^R_{[k]}(\beta,:)}\\
        &= Re\langle (\zl{R}_N)_{i_k=\alpha}, (\zl{R}_N)_{i_k=\beta} \rangle_R = 0,~~~~\alpha\neq \beta.
    \end{align*}

    \color{black}
  The above establishes the required  properties of $\zl S$ along modes $m+1,\ldots,N$.  In a similar fashion, one can also prove  that for $j=1,\ldots,m$: 1) $\|\zl{S}_{i_j=1}\|_F=(\sigma_1^L)^{(j)}\geq   \cdots \geq \|\zl{S}_{i_j=I_j}\|_F=(\sigma_{I_j}^L)^{(j)}\geq 0$, 2) $Re\langle \zl{S}_{i_j=\alpha}, \zl{S}_{i_j=\beta} \rangle_L =0$, and 3) $\langle \zl{S}_{i_1=\alpha} , \zl{S}_{i_1=\beta} \rangle_L =0,~\alpha\neq\beta$, via 1)   proving similar properties for   $\zl{L}_1$, and 2)   establish the relation between $\zl{L}_1$ and $\zl{S}$, as that for $\zl{R}_N$ and $\zl{S}$ above.  For  conciseness, we omit the detailed proof. 
    \color{black}
\end{proof}
\color{black}

\subsection{Matrix SVD perspective} We discuss an essential difference between TS-QHOSVD and L-QHOSVD in the matrix case.  Let $\jz T\in\quat{I_1\times I_2}$. It is easily seen that applying TS-QHOSVD to $\jz T$  is just the standard quaternion matrix SVD. 
On the other hand, applying  L-QHOSVD to $\jz T$ yields the decomposition in the following form:
\begin{equation}\label{eq:lqhosvd_N_2}
\jz T = \jz U_2 \times_2^L \jz U_1 \times_1^L \jz S. 
\end{equation}
If $\jz S$ is quaternion, then due to the non-commutativity of quaternion multiplications, \eqref{eq:lqhosvd_N_2}   is not equivalent to 
$
\jz U_1 \jz S\jz U_2^H~{\rm nor}~\jz U_1\jz S\jz U_2, 
$
namely,   the form  of matrix SVD.  In the following, we show that $\jz S$ may be quaternion. 
\begin{proposition}
    Let  $\jz S,\jz U_1,\jz U_2$ be generated by L-QHOSVD applied to   $\jz T\in\quat{I_1\times I_2}$. Then $\jz S$ may be quaternionic.
\end{proposition}
\begin{proof}
 According to the description above \refthm{L}, we first denote $\jz T_3:=\jz T$. Then, write $(\jz T_3)^L_{[2]}$ as the mode-$2$ unfolding of $\jz T_3$, which, according to its definition, is $(\jz T_3)^L_{[2]} = (\jz T_3)^T = \jz T^T$.   By   L-QHOSVD,  one first computes the SVD of $(\jz T_3)^L_{[2]} = \jz T^T= \jz U_2 \Sigma_2\jz V_2^H$ with $\jz U_2$ the required mode-$2$ factor matrix and $ \Sigma_2$ a real diagonal matrix. Then one computes $\jz T_2 = \jz U_2^H\times_2^L \jz T_3 = \jz U_2^H\times_2^L \jz T$, which is
 \begin{align*} 
     (\jz T_2)_{i_1 i_2} = \sum_{j_2=1}^{I_2}\nolimits(\jz U_2^H)_{i_2 j_2}(\jz T)_{i_1 j_2} = \sum_{j_2=1}^{I_2}\nolimits(\jz U_2^H)_{i_2 j_2}(\jz T^T)_{j_2 i_1} = (\jz U_2^H\jz T^T)_{i_2i_1}, 
 \end{align*}
 and so 
$    \jz T_2 = (\jz U_2^H\jz T^T)^T = (\Sigma_2\jz V_2^H)^T = \bar{\jz V}_2\Sigma_2^T$, 
 where the second equality follows from the SVD expression of $\jz T^T$, and the last one uses Proposition \ref{lem:real_quater_transpose_commutat}.  Here $\bar{\jz V}_2$ is the conjugation of $\jz V_2$ which may not be unitary.
 One next computes the SVD of $(\jz T_2)^L_{[1]} = \jz T_2 = \jz U_1\Sigma_1\jz V_1^H$ to obtain $\jz{U}_1$, and finally, the core tensor $\jz S = \jz U_1^H\times_1^L\jz T_2 = \jz U_1^H \jz T_2 = \jz U_1^H \bar{\jz V}_2\Sigma_2^T$, which may still be a quaternion matrix.
\end{proof}



\color{black}

\section{Truncated TS-HOSVD and Error Bound Analysis}\label{error}
  In this section, we consider a rank-truncated version of   TS-QHOSVD  and     establish its error bound, extending those of Algorithm \ref{algo:hosvd} and \ref{algo:sequential_hosvd}. An analogue will also be done for L-QHOSVD. 

\subsection{Truncated TS-QHOSVD}
The truncated TS-QHOSVD is executed similar to   TS-QHOSVD   (Algorithm~\ref{algo:trun_ts}), except that the   SVD is replaced by a truncated SVD. Specifically,  given truncated ranks $\{I'_1,I'_2,\dots,I'_N\}$, we first initialize $\hat{\zl L}_{m+1} = {\zl T} = \hat{\zl R}_m$, and then simultaneously launch two independent recursion procedure on two processors: 
\begin{equation}
    \begin{split}
        &\qquad ({\rm Processor~1}) \qquad\qquad\qquad~~~ ({\rm Processor~2}) \\
        &\hat{\zl{L}}_{m}  :=\hat{\jz{U}}_m^H \times_m^L \hat{\zl{L}}_{m+1} \qquad\quad~~ \hat{\zl{R}}_{m+1}  :=\hat{\zl{R}}_m \times_{m+1}^R \hat{\jz{V}}_{m+1} \\
         &\qquad \vdots  \qquad\qquad\qquad\qquad\qquad\qquad \vdots  \\
        &\hat{\zl{L}}_1  :=\hat{\jz{U}}_1^H\times_1^L \hat{\zl{L}}_2   \qquad\qquad~~~~ \hat{\zl{R}}_{N} :=\hat{\zl{R}}_{N-1} \times_{N}^R \hat{\jz{V}}_{N},
    \end{split}
    \label{eq:parallel_recursion_truncated}
\end{equation}
in which $\hat{\jz{U}}_i\in\quat{I_i\times I_i^\prime}$ consists of the left leading  $I_i^\prime$ singular vectors of $(\hat{\jz{L}}_{i+1})^L_{[i]}\in \quat{I_i \times (\prod_{\substack{s=1}}^{i-1} I_s)(\prod_{\substack{s=i+1}}^{m} I_s^\prime)(\prod_{\substack{s=m+1}}^{N} I_s)}$,  $i=m, \ldots,1$, and $\hat{\jz{V}}_j\in\quat{I_j\times I_j^\prime}$ consists of    the right leading $I_j^\prime$ singular vectors of $(\hat{\jz{R}}_{j-1})^R_{[j]}\in \quat{(\prod_{\substack{s=1}}^{m} I_s)(\prod_{\substack{s=m+1}}^{j-1} I_s^\prime)(\prod_{\substack{s=j+1}}^{N} I_s)\times I_{j}}$, $j=m+1,m+2,\ldots,N$.
Finally, we can use the left factor matrices \( \hat{\jz U}_1, \dots, \hat{\jz U}_m \) (or right factors \( \hat{\jz V}_{m+1}, \dots, \hat{\jz V}_N \), or both) to compute  the core tensor \( \hat{\zl S} \) as:
\[
\hat{\zl S} = \hat{\jz U}_1^H\times^L_1 \cdots \hat{\jz U}_m^H\times^L_m \hat{\zl R}_N = \hat{\zl L}_1\times_{m+1}^R \hat{\jz{V}}_{m+1}\cdots \times_{N}^R\hat{\jz{V}}_{N} = \hat{\jz U}^H_1 \times^L_1 \cdots    \hat{\jz U}^H_m \times^L_m \zl T\times^R_{m+1} \hat{\jz V}_{m+1} \cdots \times^R_N \hat{\jz V}_N.
\]

The procedure is summarized in Algorithm \ref{algo:trun_ts}. 
The algorithm shows that if the truncated parameters are much smaller than the size of $\zl T$, then the size of the intermediate tensors $\hat{\zl{L}}_i,\hat{\zl{R}}_j$ will be eventually smaller, which  reduces the computational complexity. This is the same as the sequential HOSVD (Algorithm \ref{algo:sequential_hosvd}).


\begin{small}
 \begin{algorithm}
    \small
    \caption{Truncation TS-QHOSVD}
    \begin{algorithmic}[1]
        \REQUIRE{ $\zl{T}\in \quat{I_1\times I_2\times\cdots \times I_N}$ and truncated parameters $\{I'_1,I'_2,\dots,I'_N\}$. Denote $\hat{\zl L}_{m+1} = {\zl T} = \hat{\zl R}_{m}$.}
        \ENSURE{$\hat{\zl{S}}\in \quat{I_1^\prime\times \cdots \times I_N^\prime}$, $\{\hat{\jz{U}}_1,\ldots ,\hat{\jz{U}}_m, \hat{\jz{V}}_{m+1},\ldots ,\hat{\jz{V}}_{N}\}.$}
    
    \FOR{$i=m:-1:1$}
        \STATE{$\hat{\jz{U}}_{i}\leftarrow $ left leading $I_i'$ singular vectors of $(\hat{\jz{L}}_{i+1})^L_{[i]}$~~~~~~~~~~~~~~~~~~~~~~~~~~~{\color{gray}\% on Processor 1}}
        \STATE{$\hat{\zl{L}}_{i}\leftarrow \hat{\jz{U}}_{i}^H \times_k^L \hat{\zl{L}}_{i+1}$~~~~~~~~~~~~~~~~~~~~~~~~~~~~~~~~~~~~~~~~~~~~~~~~~~~~~~~~~~~~~~~~~~{\color{gray}\% $\in \mathbb{H}^{I_1\times \cdots I_{i-1}\times I'_{i}\times \cdots \times I'_{m}\times I_{m+1}\times \cdots \times I_N}$ }}
    \ENDFOR
    \FOR{$j = m+1 : N$}
        \STATE{$\hat{\jz{V}}_{j}\leftarrow $ right leading $I_j'$ singular vectors of $(\hat{\jz{R}}_{j-1})^R_{[j]}$~~~~~~~~~~~~~~~~~~~~~~~~{\color{gray}\% on Processor 2}}
        \STATE{$\hat{\zl{R}}_j\leftarrow \hat{\zl{R}}_{j-1}\times_{j}^R \hat{\jz{V}}_{j}$~~~~~~~~~~~~~~~~~~~~~~~~~~~~~~~~~~~~~~~~~~~~~~~~~~~~~~~~~~~~~~~~~~{\color{gray}\% $\in \mathbb{H}^{I_1\times \cdots \times I_m \times I'_{m+1}\times \cdots \times I'_{j}\times I_{j+1}\times \cdots \times I_N}$ }  }
    \ENDFOR
    \STATE{$\hat{\zl S} \leftarrow \hat{\jz U}_1^H\times^L_1 \cdots \hat{\jz U}_m^H\times^L_m \hat{\zl R}_N$}
    \end{algorithmic}
    \label{algo:trun_ts}
 \end{algorithm}
\end{small}

 \color{black} Having been discussed in Remark \ref{rmk:lqhosvd}, \cite{miao2023quat} truncated L-QHOSVD is to first perform a full decomposition, and then it shirinks   the core tensor, which  may be more  costly than the rank-truncated fashion. On the other hand, this truncation strategy may not be suitable for   low-rank approximation. Thus we adopt rank truncation here.\color{black}

Once the core tensor and the factors are obtained, the approximate tensor can be reconstructed as:
  \begin{equation}
    \begin{split}\hat{\zl{T}} &= \hat{\jz{U}}_m\times_m^L\cdots \hat{\jz{U}}_{1} \times_1^L\hat{\zl{S}} \times_N^R \hat{\jz{V}}_N^H\dots \times_{m+1}^R\hat{\jz{V}}^H_{m+1}\\
        &= \hat{\jz{U}}_m\times_m^L\cdots \hat{\jz{U}}_{1}\times_{1}^L\left(\hat{\jz{U}}_1^H\times_1^L\cdots \hat{\jz{U}}_{m}^H\times_{m}^L\zl{T} \times_{m+1}^R \hat{\jz{V}}_{m+1}\cdots \times_N^R\hat{\jz{V}}_N\right) \times_N^R \hat{\jz{V}}_N^H\dots \times_{m+1}^R\hat{\jz{V}}^H_{m+1}\label{eq:hat_T_representation}.
    \end{split}\end{equation}
    The coming subsection shows that $\|\hat{\zl{T}} - \zl{T}\|_F$ can be bounded by the tail energy.
 
\subsection{Error analysis}
To analyze the error bound, we first present some notations in the sequel.
Recall that the left mode-$i$ unfolding of $\hat{\zl L}_{i+1}$ and right mode-$j$ unfolding of $\hat{\zl R}_{j-1}$ are respectively written as $(\hat{\jz{L}}_{i+1})^L_{[i]}$ and $(\hat{\jz{R}}_{j-1})^R_{[j]}$. For any $i=1,2,\cdots,m$ and $j = m+1,m+2,\cdots,N$, we write the   SVD  of $(\hat{\jz{L}}_{i+1})^L_{[i]}$ and $(\hat{\jz{R}}_{j-1})^R_{[j]}$ as:
\begin{equation}
    \begin{split}
         (\hat{\jz{L}}_{i+1})^L_{[i]}&=\jz{U}_i\Sigma^L_i\jz{V}^H_i = [\hat{\jz U}_i, \tilde{\jz U}_i] \Sigma^L_i\jz{V}^H_i
\in \quat{I_i \times (\prod_{\substack{s=1}}^{i-1} I_s)(\prod_{\substack{s=i+1}}^{m} I_s^\prime)(\prod_{\substack{s=m+1}}^{N} I_s)}; \\
         (\hat{\jz{R}}_{j-1})^R_{[j]}&=\jz{U}_{j}\Sigma^R_{j}\jz{V}^H_{j}=\jz{U}_{j}\Sigma^R_{j}\begin{bmatrix}
             \hat{\jz V}_{j}^H \\
             \tilde{\jz V}_{j}^H
         \end{bmatrix} 
         \in \quat{(\prod_{\substack{s=1}}^{m} I_s)(\prod_{\substack{s=m+1}}^{j-1} I_s^\prime)(\prod_{\substack{s=j+1}}^{N} I_s)\times I_{j}},
    \end{split}
    \label{equ:svd}
\end{equation}
where for the first relation, $\jz U_i = [\hat{\jz U}_i, \tilde{\jz U}_i]\in \quat{I_i\times I_i}$,  $\hat{\jz U}_i\in \quat{I_i\times I_i^\prime}$ is the required factor matrix, and $\tilde{\jz U}_i\in\quat{I_i \times (I_i-I_i^\prime) }$ represents its orthonormal complement.  For the second relation,  $\jz V_j = [\hat{\jz V}_j,\tilde{\jz V}_j] \in \quat{I_j \times I_j}$ is partitioned similarly. 


\begin{theorem}[Error bound of truncated TS-QHOSVD] \label{thm:error_bound_ts_qhosvd}
  \color{black}  For   $\zl{T}\in\quat{I_1\times \cdots \times I_N}$, 
    let $\hat{\zl{T}}$ be a  ($I'_1,I'_2,\dots,I'_N$) truncated TS-QHOSVD  of $\zl{T}$ given by
    \begin{align*}
        \hat{\zl{T}}=\hat{\jz{U}}_m\times_m^L\cdots \hat{\jz{U}}_{1}\times_{1}^L\hat{\zl{S}} \times_N^R \hat{\jz{V}}_N^H\cdots \times_{m+1}^R\hat{\jz{V}}^H_{m+1},
    \end{align*}
    where $\hat{\zl{S}},\hat{\jz{U}}_1,\cdots,\hat{\jz{U}}_m,\hat{\jz V}_{m+1},\cdots,\hat{\jz V}_N$ are computed by \refalgo{algo:trun_ts}. 
    Then we have the following error bound:
    \begin{align*}
        \|\zl{T}-\hat{\zl{T}}\|_F^2 
        \leq &\sum_{k=I'_1+1}^{I_1}((\sigma^L_{k})^{{(1)}})^2+\cdots+\sum_{k=I'_m+1}^{I_m}((\sigma^L_{k})^{{(m)}})^2 \\
        &+\sum_{k=I'_{m+1}+1}^{I_{m+1}}((\sigma^R_{k})^{{(m+1)}})^2+\cdots+\sum_{k=I'_N+1}^{I_N}((\sigma_{k}^R)^{{(N)}})^2.
    \end{align*}
    where $(\sigma^L_{k})^{(i)}$ and $(\sigma^R_{k})^{(j)}$ are the $k$-th largest singular values of $\Sigma^L_i$ and $\Sigma^R_j$ defined in \eqref{equ:svd}, respectively.
\end{theorem}

    The result  shows  that in contrast to a direct quaternion extension of the truncated HOSVD (see Example \ref{example:1}), the reconstruction error corresponding to truncated TS-QHOSVD is upper bounded by the     tail energy, which generalizes that of truncated HOSVD and sequential HOSVD \cite{hosvd,vannieuwenhoven2012new}. In the following,  we use the tensor in Example \ref{example:1} to illustrate the above fact. We set truncated ranks
  $(I_1^\prime,\!\ldots\!,I_4^\prime) = (2,2,2,2)$. After computation,   $\Sigma_1^L,\Sigma^L_2,\Sigma^R_3,\Sigma^R_4$, i.e.,   the   singular value matrices defined in \eqref{equ:svd}, are given as follows:
\begin{align*}  
    \Sigma^L_1 = {\rm diag}(5.0110,0.4822,0.0717),~\Sigma^L_2 = {\rm diag}(5.0134,0.4621,0.0605),\\
    \Sigma^R_3 = {\rm diag}(5.0134,0.4621,0.0605),~\Sigma^R_4 = {\rm diag}(5.0122, 0.4689,0.0713).
\end{align*}
One can compute the error bound    
 $
    \|\zl{T}-\hat{\zl T}\|_F^2=0.0171 \leq \sigma_3(\Sigma^L_1)^2+\sigma_3(\Sigma^L_2)^2+\sigma_3(\Sigma^R_3)^2+\sigma_3(\Sigma^R_4)^2 = 0.0175
 $.

To prove the error bound, we first need the following lemmas.
\begin{lemma}    \label{lem:ts-1}
    Given   $\zl{T}\in \quat{I_1\times \cdots \times I_N}$, let $\hat{\zl{S}}\in \quat{I'_1\times \cdots \times I'_N}$, $\hat{\jz{U}}_k\in \quat{I_k\times I'_k}$ and
    $\hat{\jz{V}}_k\in \quat{I_k\times I'_k}$ be  generated by   \refalgo{algo:trun_ts}. Then 
    \[\|\hat{\jz{U}}_k\times_k^L \hat{\zl{S}}\|_F= \|\hat{\zl{S}}\|_F,~ ~\|\hat{\jz{U}}_k^H\times_k^L \zl{T}\|_F\leq \|\zl{T}\|_F; ~~~~\|\hat{\zl{S}}\times_k^R \hat{\jz{V}}_k^H\|_F= \|\hat{\zl{S}}\|_F, ~~\|\zl{T}\times_k^R \hat{\jz{V}}_k\|_F\leq \|\zl{T}\|_F.\]
\end{lemma}
\begin{proof}
    Here we only prove the left product case, while for the right product, the proof is analogous:
    \[\|\hat{\jz{U}}_k\times_k^L \hat{\zl{S}}\|_F^2=\|\hat{\jz{U}}_k\hat{\jz{S}}_{[k]}^L\|_F^2
    =tr((\hat{\jz{S}}_{[k]}^L)^H\hat{\jz{U}}_k^H\hat{\jz{U}}_k\hat{\jz{S}}_{[k]}^L)
    =tr((\hat{\jz{S}}_{[k]}^L)^H\hat{\jz{S}}_{[k]}^L)
    =\|\hat{\jz{S}}_{[k]}\|_F^2=\|\hat{\zl{S}}\|_F^2.\]
    On the other hand,
    \begin{align*}
        \|\hat{\jz{U}}_k^H\times_k^L \zl{T}\|_F^2=\|\hat{\jz{U}}_k^H\jz{T}^L_{[k]}\|_F^2 = \|\hat{\jz{U}}_k\hat{\jz{U}}_k^H\jz{T}^L_{[k]}\|_F^2 \leq 
    \|\jz{T}^L_{[k]}\|^2_F=\|\zl{T}\|^2_F.
    \end{align*}
\end{proof}

\begin{lemma} 
    \label{lem:truncated_ts_qhosvd_lem3}
    Define the following tensors:
    \begin{align*}
        & \hat{\zl T}_R := \hat{\zl R}_N\times^R_N \hat{\jz V}_N^H \cdots \times^R_{m+1} \hat{\jz V}^H_{m+1}=(\zl{T}\times_{m+1}^R\hat{\jz{V}}_{m+1}\cdots\times_N^R\hat{\jz V}_N)\times_N^R\hat{\jz{V}}_N^H\cdots\times_{m+1}^R\hat{\jz{V}}_{m+1}^H; \\
        & \hat{\zl G}_i := \hat{\zl L}_{i+1} - \hat{\jz U}_{i}\times^L_{i}\cdots \hat{\jz U}_{1}\times^L_1 (\hat{\jz U}_1^H \times^L_1 \cdots \hat{\jz U}_m^H \times^L_m \hat{\zl T}_R),~ i = m,m-1,\ldots ,1;\\
        & \hat{\zl H}_j := \hat{\zl R}_{j-1} - \hat{\zl R}_N\times^R_N \hat{\jz V}^H_N \cdots \times^R_{j} \hat{\jz V}^H_{j},~ j = m+1,m+2,\ldots, N.
    \end{align*}
    Then the following identities hold:
    \[
        \|\hat{\zl G}_{m}\|_F^2 = \|{\zl T} - \hat{\zl T}\|_F^2,\quad {\rm and} \quad \|\hat{\zl H}_{m+1}\|_F^2 = \|{\zl T} - \hat{\zl T}_R\|_F^2.
    \]
    In particular, we define the boundary cases:
    \begin{align*}
        \hat{\zl G}_0 = \hat{\zl L}_1 - \hat{\jz U}_1^H \times_1^L \cdots \hat{\jz U}_m^H \times_m^L \hat{\zl T}_R,\quad {\rm and} \quad \hat{\zl H}_{N+1} = \hat{\zl R}_N - \hat{\zl R}_N = 0.
    \end{align*}
    Then we have:
    \begin{equation}\label{equ:hat_G0_bound}
        \begin{split}
            & \|\hat{\zl G}_0\|_F^2 = \|\hat{\jz U}_1^H \times_1^L \cdots \hat{\jz U}_m^H \times_m^L ({\zl T} - \hat{\zl T}_R)\|_F^2 \leq \|({\zl T} - \hat{\zl T}_R)\|_F^2 = \|\hat{\zl H}_{m+1}\|_F^2.
        \end{split}
    \end{equation}
\end{lemma}
\begin{proof}
    From the initialization $\hat{\zl L}_{m+1} = {\zl T} = \hat{\zl R}_m$ and the definitions of $\hat{\zl G}_i$, $\hat{\zl H}_j$, and $\hat{\zl T}$ we have:
    \begin{align*}
        &\|\hat{\zl G}_m\|_F^2 = \|\hat{\zl L}_{m+1} - \hat{\jz U}_m\times^L_m \cdots \hat{\jz U}_1\times^L_1 (\hat{\jz U}_1^H \times^L_1 \cdots \hat{\jz U}_m^H \times^L_m \hat{\zl T}_R)\|_F^2 = \|\hat{\zl L}_{m+1} - \hat{\zl T}\|_F^2 = \|{\zl T} - \hat{\zl T}\|_F^2,\\
        &\|\hat{\zl H}_{m+1}\|_F^2 = \|\hat{\zl R}_m - \hat{\zl R}_N\times^R_N \hat{\jz V}^H_N\cdots\times^R_{m+1} \hat{\jz V}^H_{m+1}\|_F^2 = \|{\zl T} - \hat{\zl T}_R\|_F^2.
    \end{align*}
    For the boundary case of $\hat{\zl G}_0$, using the definition $\hat{\zl L}_1 = \hat{\jz U}_1^H \times_1^L \cdots \hat{\jz U}_m^H \times_m^L {\zl T}$ we have:
    \begin{align*}
        \|\hat{\zl G}_0\|_F^2 =& \|\hat{\zl L}_1 - \hat{\jz U}_1^H \times_1^L \cdots \hat{\jz U}_m^H \times_m^L {\zl T}_R\|_F^2 = \|\hat{\jz U}_1^H \times_1^L \cdots \hat{\jz U}_m^H \times_m^L {\zl T} - \hat{\jz U}_1^H \times_1^L \cdots \hat{\jz U}_m^H \times_m^L {\zl T}_R\|_F^2\\
        =& \|\hat{\jz U}_1^H \times_1^L \cdots \hat{\jz U}_m^H \times_m^L ({\zl T}  - \hat{\zl T}_R)\|_F^2 \leq \|({\zl T}  - \hat{\zl T}_R)\|_F^2= \|\hat{\zl H}_{m+1}\|_F^2.
    \end{align*}
    where the inequality follows from Lemma \ref{lem:ts-1}. 
\end{proof}
\begin{proof}[Proof of \refthm{thm:error_bound_ts_qhosvd}]
    For $i = m, m-1, \dots, 1$,  we derive:
    \begin{align*}
        \|\hat{\zl G}_i\|_F^2 = & \|\hat{\zl L}_{i+1} - \hat{\jz U}_{i}\times^L_{i}\cdots \hat{\jz U}_{1}\times^L_1 \hat{\jz U}_1^H \times^L_1 \cdots \hat{\jz U}_m^H \times^L_m \hat{\zl T}_R\|_F^2 \\
            = & \|\hat{\zl L}_{i+1} {\color{blue}- \hat{\jz U}_{i}\hat{\jz U}_{i}^H \times_{i}^L\hat{\zl L}_{i+1} + \hat{\jz U}_{i}\hat{\jz U}_{i}^H \times_{i}^L\hat{\zl L}_{i+1}} - \hat{\jz U}_{i}\times^L_{i}\cdots \hat{\jz U}_m\times^L_m \hat{\jz U}_m^H \times^L_m \cdots \hat{\jz U}_1^H \times^L_1 \hat{\zl T}_R\|_F^2 \\
            = & \|({\jz I} - \hat{\jz U}_{i}\hat{\jz U}_{i}^H) \times_{i}^L\hat{\zl L}_{i+1} + \hat{\jz U}_{i} \times^L_{i}(\hat{\zl L}_{i} - \hat{\jz U}_{i-1}\times^L_{i-1}\cdots \hat{\jz U}_m\times^L_m \hat{\jz U}_m^H \times^L_m \cdots \hat{\jz U}_1^H \times^L_1 \hat{\zl T}_R)\|_F^2 \\
            = & \|({\jz I} - \hat{\jz U}_{i}\hat{\jz U}_{i}^H) \times_{i}^L\hat{\zl L}_{i+1} + \hat{\jz U}_{i} \times^L_{i}\hat{\zl G}_{i-1}\|_F^2 = \|({\jz I} - \hat{\jz U}_{i}\hat{\jz U}_{i}^H) \times_{i}^L\hat{\zl L}_{i+1}\|_F^2 + \|\hat{\jz U}_{i} \times^L_{i}\hat{\zl G}_{i-1}\|_F^2 \\
             = & 
            \sum_{k = I_i^\prime+1}^{I_i}\nolimits \bigxiaokuohao{(\sigma^L_{k})^{(i)}}^2 + \|\hat{\zl G}_{i-1}\|_F^2,
    \end{align*}
    where the third equality uses the recursive formula $\hat{\zl L}_{i} = \hat{\jz U}_i^H \times^L_{i} \hat{\zl L}_{i+1}$ in \eqref{eq:parallel_recursion_truncated}, the fourth one is from the definition of $\hat{\zl G}_i$, the fifth one uses Pythagorean theorem,   and the last one uses Lemma \ref{lem:ts-1} and that   $\|({\jz I} - \hat{\jz U}_{i}\hat{\jz U}_{i}^H) \times_{i}^L\hat{\zl L}_{i+1}\|_F^2=  \|({\jz I} - \hat{\jz U}_{i}\hat{\jz U}_{i}^H)(\hat{\jz L}_{i+1})^L_{[i]}\|_F^2$ is just the tail energy of $(\hat{\jz L}_{i+1})^L_{[i]}$, namely,  $ \sum_{k = I_{i}^\prime+1}^{I_i} \bigxiaokuohao{(\sigma^L_{k})^{(i)}}^2$. 
    

    For the right recursion on $\hat{\zl{H}}_j$ for $j = m+1, \dots, N$, a similar fashion yields: 
    \begin{align*}
        \|\hat{\zl H}_j\|_F^2 =& \|\hat{\zl R}_{j-1} - \hat{\zl R}_N\times^R_N \hat{\jz V}^H_N \cdots \times^R_{j} \hat{\jz V}^H_{j}\|_F^2\\
        = & \|\hat{\zl R}_{j-1} {\color{blue}- \hat{\zl R}_{j-1}\times^R_j \hat{\jz V}_j\hat{\jz V}^H_j + \hat{\zl R}_{j-1}\times^R_j \hat{\jz V}_j\hat{\jz V}^H_j} - \hat{\zl R}_N\times^R_N \hat{\jz V}^H_N \cdots \times^R_{j} \hat{\jz V}^H_{j}\|_F^2\\
        = & \|\hat{\zl R}_{j-1}\times^R_j(I - \hat{\jz V}_j\hat{\jz V}^H_j) + (\hat{\zl R}_{j} - \hat{\zl R}_N\times^R_N \hat{\jz V}^H_N \cdots \times^R_{j+1} \hat{\jz V}^H_{j+1})\times^R_{j} \hat{\jz V}^H_{j}\|_F^2 \\  
        = & \|(\hat{\jz R}_{j-1})^R_{[j]}(I - \hat{\jz V}_j\hat{\jz V}^H_j)\|_F^2 + \|(\hat{\jz H}_{[j+1]})^R_{[j]}\hat{\jz V}^H_j\|_F^2 
        =  \sum_{i_j = I'_j+1}^{I_j}\nolimits((\sigma^R_{i_j})^{(j)})^2 + \|\hat{\zl H}_{j+1}\|_F^2,
    \end{align*}
where the fourth equality uses the definition of $\hat{\zl H}_{t+1}$.


 With the above two recursions  and noting  $\|{\zl T} - \hat{\zl T}\|_F^2  =  \|\hat{\zl G}_m\|_F^2;~\|\hat{\zl G}_0\|_F^2 \leq \|\hat{\zl H}_{m+1}\|_F^2;~\hat{\zl H}_{N+1}=0 $ in Lemma \ref{lem:truncated_ts_qhosvd_lem3},  
    \begin{align*}
        \|{\zl T} - \hat{\zl T}\|_F^2  =& \|\hat{\zl G}_m\|_F^2
        \leq  \sum_{k=I'_m+1}^{I_m}\nolimits((\sigma^L_{k})^{{(m)}})^2 + \cdots + \sum_{k=I'_1+1}^{I_1}\nolimits((\sigma^L_{k})^{{(1)}})^2 + \|\hat{\zl H}_{m+1}\|_F^2 \\
        = &  
           \sum_{k=I'_m+1}^{I_m}\nolimits((\sigma^L_{k})^{{(m)}})^2 + \cdots + \sum_{k=I'_1+1}^{I_1}\nolimits((\sigma^L_{k})^{{(1)}})^2
         +\sum_{k=I'_{m+1}+1}^{I_{m+1}}\nolimits((\sigma^R_{k})^{{(m+1)}})^2+\cdots+\sum_{k=I'_N+1}^{I_N}\nolimits((\sigma_{k}^R)^{{(N)}})^2.
    \end{align*}
    \end{proof}

\begin{remark}
       The proof   is  different from the real/complex truncated HOSVDs \cite{hosvd,vannieuwenhoven2012new}, where  the reconstructed tensor  $\hat{\zl T}$ can be written as
    \begin{align*}
        \hat{\zl T} &= \hat{\zl S} \times_1 \hat{\jz U}_1 \cdots \times_N \hat{\jz U}_N  = \zl T\times_1 \hat{\jz U}_1^H \cdots \times_N \hat{\jz U}_N^H \times_1\hat{\jz U}_1\cdots\times_N\hat{\jz U}_N\\
        &= \zl T\times_1 (\hat{\jz U}_1\hat{\jz U}_1^H) \cdots\times_N (\hat{\jz U}_N\hat{\jz U}_N^H),
    \end{align*}
    which is a multilinear  projection of $\zl T$ along each mode.  
    With this representation, one   can     bound $\|\zl T-\hat{\zl T}\|_F^2$ by decomposed it as $\sum^N_{k=1}\|\zl T-\zl T\times_k(\hat{\jz U}_k\hat{\jz U}_k^H )\|_F^2$ and then obtain the error bound directly.  However, this multilinear projection representation is no longer valid in the quaternion case due to its non-commutativity in multiplicatitons.
\end{remark}

\subsection{Truncated L-QHOSVD}
Different from \cite{miao2023quat}, we can also derive a rank truncated version of   L-QHOSVD, which is presented in Algorithm \ref{algo:stl}. The error bound  analysis is analogous to that of Theorem \ref{thm:error_bound_ts_qhosvd} and is omitted.
\begin{small}
\begin{algorithm}
    \small
    \caption{Truncated L-QHOSVD}
    \begin{algorithmic}[1]
        \REQUIRE{ Quaternion tensor $\zl{T}\in \quat{I_1\times I_2\times\cdots \times I_N}$ and truncated parameters $\{I'_1,I'_2,\dots,I'_N\}$. Let $\hat{\zl T}_{N+1} = \zl T$.}
        \ENSURE{$\hat{ \zl{S}}$, $\{\hat{\jz{U}}_1, \hat{\jz{U}}_2, \cdots ,\hat{\jz{U}}_N\}.$}
    
    \FOR{$k=N:-1:1$}
        \STATE{$\hat{\jz{U}}_{k}\leftarrow $ Left leading $I_k'$ singular vectors of $(\hat{\jz{T}}_{k+1})^L_{[k]}$}
        \STATE{$\hat{\zl{T}}_{k}\leftarrow \hat{\jz{U}}_{k}^H \times_k^L \hat{\zl{T}}_{k+1}$}
    \ENDFOR
    \STATE{$\hat{ \zl{S}}\leftarrow \hat{\zl{T}_1}$}
    \end{algorithmic}     \label{algo:stl}
 \end{algorithm}
\end{small}


\begin{theorem}[Error bound of  truncated L-QHOSVD] \label{thm:error_bound_l_qhosvd}
    For  $\zl{T}\in \quat{I_1 \times \cdots \times I_N}$, let $\hat{\zl{T}}$ be a ($I'_1,I'_2,\dots,I'_N$)   truncated L-QHOSVD of $\zl{T}$, i.e.,
     $\hat{\zl{T}}=\hat{\jz{U}}_N\times_N^L\dots \hat{\jz{U}}_2\times_2^L\hat{\jz{U}}_{1}\times_{1}^L\hat{\zl{S}}$, 
    where $\hat{\zl{S}},\hat{\jz{U}}_1,\cdots,\hat{\jz{U}}_N$ are   computed by \refalgo{algo:stl}.
    Then   
    \begin{align}
        \|\zl{T}-\hat{\zl{T}}\|_F^2\leq \sum_{k=I'_1+1}^{I_1}\nolimits((\sigma_{k}^L)^{(1)})^2+\sum_{k=I'_2+1}^{I_2}\nolimits((\sigma_{k}^L)^{(2)})^2+\cdots+\sum_{k=I'_N+1}^{I_N}\nolimits((\sigma_{k}^L)^{(N)})^2,
        \label{eq:truncated_l_qhosvd_model}
    \end{align}
    where $(\sigma^L_{i_k})^{(k)}$ is the $i_k$-th largest singular value of $(\hat{\jz T}_{k+1})^L_{[k]}$.
\end{theorem}

\section{Numerical Experiments}\label{examples}
In this section, we   conducted   experiments on color video denoising,  color video compression, as well as scientific data  to show the efficacy of TS-QHOSVD.
All experiments are conducted on a desktop computer with an Intel i7-13700KF CPU and 32GB of RAM. The supporting software is MATLAB 2024a. 
For TS-QHOSVD with odd-order tensors, left factor $\jz{U}_i$'s correspond to the first $(n+1)/2$ modes, with the remaining $(n-1)/2$ modes for right factor $\jz{V}_j$'s. Performance is evaluated using relative reconstruction error (${\rm Re.Err} = \| \zl{T} - \hat{\zl{T}} \|_F/\| \zl{T} \|_F$) and elapsed time, where $\zl{T}$ denotes the original data tensor and $\hat{\zl{T}}$ represents the reconstructed one.

\paragraph{Example 1: Synthetic low-rank approximation}
\begin{figure}[htbp]
    \centering
%
%
\definecolor{mycolor1}{rgb}{0.87059,0.34510,0.16863}%
\definecolor{mycolor2}{rgb}{0.09412,0.40784,0.69804}%
\begin{tikzpicture}[scale=.8]

\begin{axis}[%
width=2.8in,
height=2.4in,
at={(0.5in,0.481in)},
scale only axis,
xmin=8,
xmax=47,
xminorticks=true,
xlabel style={font=\color{white!15!black}},
xlabel={r:},
ymin=0.2,
ymax=1.4,
yminorticks=true,
ylabel style={font=\color{white!15!black}},
ylabel={Time},
axis background/.style={fill=white},
legend style={at={(0.03,0.97)}, anchor=north west, legend cell align=left, align=left, draw=white!15!black}
]
\addplot [color=mycolor1, line width=2pt, mark=triangle, mark options={solid, mycolor1}]
  table[row sep=crcr]{%
10	0.36509033\\
11	0.37222502\\
12	0.36322542\\
13	0.39181613\\
14	0.39637901\\
15	0.40347514\\
16	0.41111563\\
17	0.42909325\\
18	0.44864719\\
19	0.45854096\\
20	0.45459976\\
21	0.47268344\\
22	0.48561608\\
23	0.50850871\\
24	0.50924397\\
25	0.51626571\\
26	0.53866471\\
27	0.53582287\\
28	0.55656028\\
29	0.57514416\\
30	0.56968742\\
31	0.59304907\\
32	0.60527744\\
33	0.62049092\\
34	0.64329369\\
35	0.65585806\\
36	0.66740288\\
37	0.70065006\\
38	0.72947156\\
39	0.72956807\\
40	0.74815406\\
41	0.77049457\\
42	0.81431015\\
43	0.82068518\\
44	0.83707379\\
45	0.87798693\\
};
\addlegendentry{TS-QHOSVD}

\addplot [color=mycolor2, line width=2pt, mark=square, mark options={solid, mycolor2}]
  table[row sep=crcr]{%
10	0.3881532\\
11	0.3954811\\
12	0.40529234\\
13	0.41055793\\
14	0.43894197\\
15	0.45794296\\
16	0.47866007\\
17	0.489219\\
18	0.49984079\\
19	0.52912122\\
20	0.53984313\\
21	0.55988017\\
22	0.57937206\\
23	0.60070259\\
24	0.61682942\\
25	0.62927953\\
26	0.65606818\\
27	0.68880226\\
28	0.69801316\\
29	0.71485422\\
30	0.76386805\\
31	0.77252701\\
32	0.79953574\\
33	0.81914697\\
34	0.84689352\\
35	0.8865604\\
36	0.89942172\\
37	0.97042845\\
38	0.99119904\\
39	1.02815047\\
40	1.04544165\\
41	1.08767264\\
42	1.12016752\\
43	1.15084692\\
44	1.18973484\\
45	1.2364993\\
};
\addlegendentry{L-QHOSVD}

\end{axis}

\begin{axis}[%
width=2.8in,
height=2.4in,
at={(4.2in,0.481in)},
scale only axis,
xmin=8,
xmax=47,
xminorticks=true,
xlabel style={font=\color{white!15!black}},
xlabel={r:},
ymin=0,
ymax=1,
yminorticks=true,
ylabel style={font=\color{white!15!black}},
ylabel={Relative Error},
axis background/.style={fill=white},
legend style={legend cell align=left, align=left, draw=white!15!black}
]
\addplot [color=mycolor1, line width=2pt, mark=triangle, mark options={solid, mycolor1}]
  table[row sep=crcr]{%
10	0.889621047519551\\
11	0.865915174258559\\
12	0.840994460087606\\
13	0.815129654765143\\
14	0.788206054764089\\
15	0.759293949559049\\
16	0.729049347993618\\
17	0.698957571405803\\
18	0.667772356797302\\
19	0.635484072488151\\
20	0.603304960068454\\
21	0.57135851664219\\
22	0.539524508009277\\
23	0.507477225390353\\
24	0.476397248075096\\
25	0.444990679673833\\
26	0.413410271207581\\
27	0.381758825916723\\
28	0.350822775995416\\
29	0.320041682564314\\
30	0.289358894701821\\
31	0.25949066391398\\
32	0.229682987943789\\
33	0.201243291945924\\
34	0.172206425331779\\
35	0.145022487156237\\
36	0.115320282736174\\
37	0.0876959530234014\\
38	0.0602828038588681\\
39	0.0328620155195691\\
40	1.95302830619946e-06\\
41	1.88135997467564e-06\\
42	1.80137157704836e-06\\
43	1.71153663626908e-06\\
44	1.60986057303594e-06\\
45	1.49383758909326e-06\\
};
\addlegendentry{TS-QHOSVD}

\addplot [color=mycolor2, line width=2pt, mark=square, mark options={solid, mycolor2}]
  table[row sep=crcr]{%
10	0.945165022973472\\
11	0.933810731614167\\
12	0.919733199820569\\
13	0.902480599909295\\
14	0.88198619909997\\
15	0.860309453388282\\
16	0.837232974554503\\
17	0.811428862261394\\
18	0.783351441597462\\
19	0.752305228417609\\
20	0.721353672981747\\
21	0.68926514368423\\
22	0.654335257098241\\
23	0.618869677421983\\
24	0.58305319586778\\
25	0.547148758138222\\
26	0.510130895891475\\
27	0.473518759340269\\
28	0.436269626328786\\
29	0.398590412963196\\
30	0.361822071405618\\
31	0.324844829372924\\
32	0.288009429707837\\
33	0.251733451076558\\
34	0.216968653505796\\
35	0.182502783613485\\
36	0.145896622292465\\
37	0.112019390354497\\
38	0.0779038958249866\\
39	0.0446534335174038\\
40	1.94345109156702e-06\\
41	1.87096751883017e-06\\
42	1.79001602002581e-06\\
43	1.69905714231455e-06\\
44	1.59610204308536e-06\\
45	1.4784972851503e-06\\
};
\addlegendentry{L-QHOSVD}

\end{axis}
\end{tikzpicture}%
    \captionsetup{aboveskip=-1pt, belowskip=-4pt}
    \caption{Comparison of elapsed time and reconstruction error for truncated TS-QHOSVD (Alg. \ref{algo:trun_ts})  and truncated L-QHOSVD (Alg. \ref{algo:stl}) on synthetic quaternion tensor   with varying truncation ranks $[r,r,r,r]$.}
    \label{fig:synthetic_tensor}
\end{figure}
\begin{figure}[htbp]
    \centering
%
%
\definecolor{mycolor1}{rgb}{0.87059,0.34510,0.16863}%
\definecolor{mycolor2}{rgb}{0.09412,0.40784,0.69804}%
\begin{tikzpicture}[scale=.8]

\begin{axis}[%
width=2.8in,
height=2.4in,
at={(0.5in,0.481in)},
scale only axis,
xmin=8,
xmax=47,
xminorticks=true,
xlabel style={font=\color{white!15!black}},
xlabel={r:},
ymin=0.2,
ymax=1.4,
yminorticks=true,
ylabel style={font=\color{white!15!black}},
ylabel={Time},
axis background/.style={fill=white},
legend style={at={(0.03,0.97)}, anchor=north west, legend cell align=left, align=left, draw=white!15!black}
]
\addplot [color=mycolor1, line width=2pt, mark=triangle, mark options={solid, mycolor1}]
  table[row sep=crcr]{%
10	0.37066772\\
11	0.37721214\\
12	0.37818426\\
13	0.39228994\\
14	0.40649232\\
15	0.41391099\\
16	0.42645221\\
17	0.42957513\\
18	0.44536134\\
19	0.45441325\\
20	0.47039646\\
21	0.4801227\\
22	0.48926044\\
23	0.49488963\\
24	0.50358858\\
25	0.52789568\\
26	0.54554837\\
27	0.55092162\\
28	0.5625629\\
29	0.57103269\\
30	0.58743783\\
31	0.60143088\\
32	0.62813754\\
33	0.63724088\\
34	0.6442649\\
35	0.66321459\\
36	0.6788663\\
37	0.70384285\\
38	0.71644085\\
39	0.73466115\\
40	0.74674292\\
41	0.77850267\\
42	0.7962749\\
43	0.83363127\\
44	0.86755511\\
45	0.91151432\\
};
\addlegendentry{TS-QHOSVD}

\addplot [color=mycolor2, line width=2pt, mark=square, mark options={solid, mycolor2}]
  table[row sep=crcr]{%
10	0.40063453\\
11	0.41501898\\
12	0.42052832\\
13	0.43207151\\
14	0.46333504\\
15	0.47305376\\
16	0.48184974\\
17	0.50581617\\
18	0.52276189\\
19	0.53822601\\
20	0.55432683\\
21	0.57611266\\
22	0.59261923\\
23	0.61171802\\
24	0.6345512\\
25	0.65845571\\
26	0.67986525\\
27	0.71087636\\
28	0.72728673\\
29	0.75228931\\
30	0.77221475\\
31	0.80344065\\
32	0.82904578\\
33	0.86044477\\
34	0.88273146\\
35	0.91165313\\
36	0.94956214\\
37	1.00989677\\
38	1.02082899\\
39	1.05829712\\
40	1.06691859\\
41	1.16267856\\
42	1.18782731\\
43	1.21866989\\
44	1.26557582\\
45	1.33666065\\
};
\addlegendentry{L-QHOSVD}

\end{axis}

\begin{axis}[%
width=2.8in,
height=2.4in,
at={(4.2in,0.481in)},
scale only axis,
xmin=8,
xmax=47,
xminorticks=true,
xlabel style={font=\color{white!15!black}},
xlabel={r:},
ymin=0,
ymax=1,
yminorticks=true,
ylabel style={font=\color{white!15!black}},
ylabel={Relative Error},
axis background/.style={fill=white},
legend style={legend cell align=left, align=left, draw=white!15!black}
]
\addplot [color=mycolor1, line width=2pt, mark=triangle, mark options={solid, mycolor1}]
  table[row sep=crcr]{%
10	0.976972295307882\\
11	0.969873930257192\\
12	0.961014368787541\\
13	0.950114068184835\\
14	0.93819551373014\\
15	0.923637680975697\\
16	0.907308879976146\\
17	0.888437391684093\\
18	0.866820297194618\\
19	0.843646909951498\\
20	0.81779672469136\\
21	0.799071878563034\\
22	0.779030038812756\\
23	0.758536012836961\\
24	0.737154822140384\\
25	0.714986434958714\\
26	0.69270029805598\\
27	0.670106532740184\\
28	0.646102264449644\\
29	0.621925555767054\\
30	0.597611816136211\\
31	0.572514743400275\\
32	0.546973752942064\\
33	0.521129838586511\\
34	0.494747334790329\\
35	0.468458349180405\\
36	0.441613146304681\\
37	0.41484926163192\\
38	0.387685780706723\\
39	0.360611061913545\\
40	0.332639534899112\\
41	0.304290652993189\\
42	0.275332856184776\\
43	0.246400784274371\\
44	0.218585447897392\\
45	0.189408596637907\\
};
\addlegendentry{TS-QHOSVD}

\addplot [color=mycolor2, line width=2pt, mark=square, mark options={solid, mycolor2}]
  table[row sep=crcr]{%
10	0.984509908313395\\
11	0.980394971314022\\
12	0.974859863193918\\
13	0.968684891926986\\
14	0.961629656860424\\
15	0.95339165730787\\
16	0.943503231817987\\
17	0.933421514864545\\
18	0.922288670559346\\
19	0.909682427259966\\
20	0.895779810817833\\
21	0.880396905640957\\
22	0.863116638092981\\
23	0.846032180692646\\
24	0.826970657860347\\
25	0.807264739675754\\
26	0.786190640193153\\
27	0.763942035957825\\
28	0.740999513016787\\
29	0.717002931518681\\
30	0.692257360376511\\
31	0.666530378534001\\
32	0.639511225787843\\
33	0.612102462408743\\
34	0.583901522356194\\
35	0.554642993908297\\
36	0.52515294740119\\
37	0.494819420178127\\
38	0.464215227308178\\
39	0.433320066118284\\
40	0.401644260768543\\
41	0.369384755925256\\
42	0.335103724879289\\
43	0.300976559030322\\
44	0.268059229287199\\
45	0.233262032166459\\
};
\addlegendentry{L-QHOSVD}

\end{axis}
\end{tikzpicture}%
    \captionsetup{aboveskip=-1pt, belowskip=-4pt}
    \caption{Comparison of elapsed time and reconstruction error for truncated TS-QHOSVD (Alg. \ref{algo:trun_ts})  and truncated L-QHOSVD (Alg. \ref{algo:stl}) on synthetic quaternion tensor   with varying truncation ranks $[r,r,r,r]$.}
    \label{fig:synthetic_tensor_rank_20}
\end{figure} 
In this example, we construct a      tensor $\zl{T} \in \mathbb{H}^{50 \times 50 \times 50 \times 50}$ as follows: we generate four  random column-orthonormal quaternion matrices $U_s = [u_{s,1},\cdots,u_{s,R}] \in \mathbb{H}^{50 \times R}$ for $s = 1,2,3,4$, and define the tensor as  
\[\zl{T} = \sum_{t = 1}^{R}\nolimits u_{1,t}  \circ u_{2,t}  \circ u_{3,t}  \circ u_{4,t}  + \alpha {\zl E} ,\]
where $\mathcal{E} \in \mathbb{H}^{50\times50\times50\times50}$ represents noise with every element drawn from standard Gaussian distribution, and $\alpha = 1/50^4$ controls noise amplitude. We set $R=10$ and $R = 20$. We compare truncated TS-QHOSVD (Alg. \ref{algo:trun_ts}) with truncated L-QHOSVD (Alg. \ref{algo:stl}) across target multilinear ranks $[r,r,r,r]$, with $r$ ranging from $10$ to $45$. The experimental results, including relative approximation error and elapsed time under various target ranks,   are shown in Fig. \ref{fig:synthetic_tensor} for $R=10$ and Fig. \ref{fig:synthetic_tensor_rank_20} for $R=20$. We   observe that TS-QHOSVD achieves lower reconstruction errors compared to L-QHOSVD across all target ranks. Concerning time, truncated TS-QHOSVD  also performs better, which   stems from its ability to compute the left and right factor matrices in parallel. The time is reduced by $22\%$-$37\%$   as truncation ranks increase.

\paragraph{Example 2: Video compression} 
\begin{figure}[htbp]
    \centering
    \includegraphics[width=\linewidth]{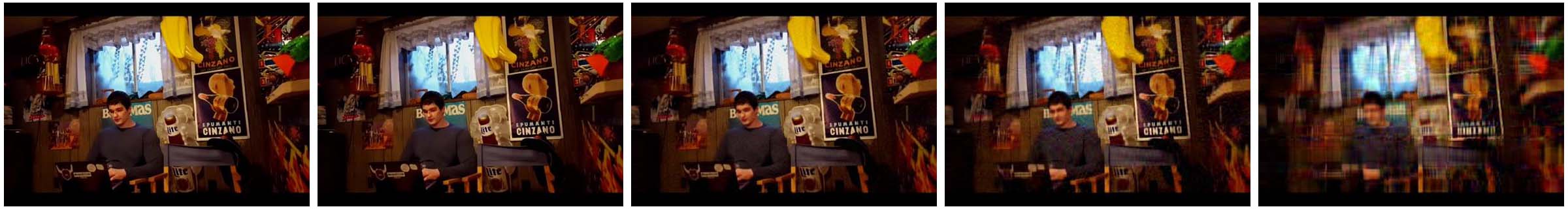}
    \captionsetup{aboveskip=-1pt, belowskip=-4pt}
    \caption{One   frame  from the   video ``5IOQB.mp4''. The first frame: the original frame; second-to-last frames: frames reconstructed by truncated TS-QHOSVD   with truncation ranks   $\alpha\cdot [320,480,500]$, with $\alpha$ varing in $\{0.5, 0.2, 0.1, 0.02\}$.}
    \label{fig:compress_video_1}
\end{figure}\begin{figure}[htbp]
    \centering
%
%
\definecolor{mycolor1}{rgb}{0.87059,0.34510,0.16863}%
\definecolor{mycolor2}{rgb}{0.09412,0.40784,0.69804}%
\definecolor{mycolor3}{rgb}{0.95294,0.63922,0.19608}%
\begin{tikzpicture}[scale=.8]

\begin{axis}[%
width=2.8in,
height=2.4in,
at={(0.5in,0.481in)},
scale only axis,
xmin=0,
xmax=0.53,
xminorticks=true,
xlabel style={font=\color{white!15!black}},
xlabel={truncated-rank ratios:},
ymin=6,
ymax=19,
yminorticks=true,
ylabel style={font=\color{white!15!black}},
ylabel={Time},
axis background/.style={fill=white},
legend style={at={(0.03,0.97)}, anchor=north west, legend cell align=left, align=left, draw=white!15!black}
]
\addplot [color=mycolor1, line width=2pt, mark=triangle, mark options={solid, mycolor1}]
  table[row sep=crcr]{%
0.02	6.62183123\\
0.04	6.89984287\\
0.06	7.16206288\\
0.08	7.37686256\\
0.1	7.41935593\\
0.12	7.68280532\\
0.14	7.92603039\\
0.16	8.02298426\\
0.18	8.23957985\\
0.2	8.69658277\\
0.22	8.64732536\\
0.24	8.99209625\\
0.26	9.18280242\\
0.28	9.59496758\\
0.3	9.72468612\\
0.32	9.90948509\\
0.34	10.23120998\\
0.36	10.0382567\\
0.38	10.10605602\\
0.4	10.55134239\\
0.42	10.74320693\\
0.44	10.91249953\\
0.46	11.16998919\\
0.48	11.69116804\\
0.5	11.76116284\\
};
\addlegendentry{TS-QHOSVD}

\addplot [color=mycolor2, line width=2pt, mark=square, mark options={solid, mycolor2}]
  table[row sep=crcr]{%
0.02	7.23560521\\
0.04	7.53491999\\
0.06	7.9123293\\
0.08	8.08580901\\
0.1	8.23134641\\
0.12	8.55692033\\
0.14	8.68492143\\
0.16	9.15316543\\
0.18	9.43518979\\
0.2	10.13115812\\
0.22	10.15931234\\
0.24	10.51581081\\
0.26	10.92447036\\
0.28	11.43298806\\
0.3	11.79797748\\
0.32	11.88478143\\
0.34	12.52890607\\
0.36	12.34463923\\
0.38	12.77783201\\
0.4	13.12188085\\
0.42	13.58822003\\
0.44	13.85376971\\
0.46	14.37548885\\
0.48	15.00045891\\
0.5	15.29738396\\
};
\addlegendentry{L-QHOSVD}


\end{axis}

\begin{axis}[%
width=2.8in,
height=2.4in,
at={(4.2in,0.481in)},
scale only axis,
xmin=0,
xmax=0.53,
xticklabel style={/pgf/number format/fixed},
xminorticks=true,
xlabel style={font=\color{white!15!black}},
xlabel={truncated-rank ratios:},
ymin=0,
ymax=0.35,
yticklabel style={/pgf/number format/fixed},
yminorticks=true,
ylabel style={font=\color{white!15!black}},
ylabel={Relative Error},
axis background/.style={fill=white},
legend style={legend cell align=left, align=left, draw=white!15!black}
]
\addplot [color=mycolor1, line width=2pt, mark=triangle, mark options={solid, mycolor1}]
  table[row sep=crcr]{%
0.02	0.278617087483099\\
0.04	0.204210110130896\\
0.06	0.161288018789825\\
0.08	0.137230199695058\\
0.1	0.117961721022279\\
0.12	0.104782325458286\\
0.14	0.0937150661513864\\
0.16	0.083470689550021\\
0.18	0.0756512345091055\\
0.2	0.0692536118078867\\
0.22	0.0627910090933972\\
0.24	0.05785029495374\\
0.26	0.052976213077365\\
0.28	0.0490909409196835\\
0.3	0.0452317247695673\\
0.32	0.0421941298844269\\
0.34	0.0393281315300164\\
0.36	0.0365555542789903\\
0.38	0.0341874091720118\\
0.4	0.0317667133842218\\
0.42	0.0298599806184066\\
0.44	0.028024425306189\\
0.46	0.0262081660587061\\
0.48	0.0246252026012579\\
0.5	0.0229822864465164\\
};
\addlegendentry{TS-QHOSVD}

\addplot [color=mycolor2, line width=2pt, mark=square, mark options={solid, mycolor2}]
  table[row sep=crcr]{%
0.02	0.274279155874415\\
0.04	0.202540713256144\\
0.06	0.161279072599682\\
0.08	0.136163588450235\\
0.1	0.118096642676931\\
0.12	0.104989922215041\\
0.14	0.0928729699759952\\
0.16	0.0828227303282066\\
0.18	0.0751622883974763\\
0.2	0.0688442731287014\\
0.22	0.0622135086787614\\
0.24	0.0573869512284895\\
0.26	0.0526503940611833\\
0.28	0.0487971657212014\\
0.3	0.0450542204609303\\
0.32	0.0419527907369188\\
0.34	0.0393542122768929\\
0.36	0.0364737308021273\\
0.38	0.034090067032398\\
0.4	0.0317197379616005\\
0.42	0.0297820062379783\\
0.44	0.0279666045317088\\
0.46	0.0261992165245506\\
0.48	0.0245439944860883\\
0.5	0.0229630797201619\\
};
\addlegendentry{L-QHOSVD}


\end{axis}
\end{tikzpicture}%
    \captionsetup{aboveskip=-1pt, belowskip=-4pt}
    \caption{Comparison of elapsed time and reconstruction error for   truncated TS-HOSVD (Alg. \ref{algo:trun_ts}) and truncated L-HOSVD (Alg. \ref{algo:stl}) with varying truncation ranks  on color video ``5IOQB.mp4'' compression.}
    \label{fig:video1}
\end{figure}

We   test  truncated TS-QHOSVD (Alg. \ref{algo:trun_ts}) and truncated L-QHOSVD (Alg. \ref{algo:stl}) in this example. 
The  test videos ``5IOQB.mp4'' and ``8JTF4.mp4'' were downloaded from the Charades dataset website \url{https://prior.allenai.org/projects/charades} (scaled to $480$p, $13$ GB).
  ``5IOQB.mp4''  consists of $854$ frames, each of size $320\times 480$.
We use 500 frames, resulting into a pure quaternion tensor in $\quat{320\times 480\times 500}$. ``8JTF4.mp4''   has $831$ frames, each of size $360\times 480$, and we use 500 frames, resulting into  a pure quaternion tensor in $\quat{360\times 480\times 500}$.   In the experiment, we set truncation ranks $\alpha\cdot[320,480,500]$, with $\alpha$ varying from $0.02$ to $0.5$.

\begin{figure}[htbp]
    \centering
    \includegraphics[width=\linewidth]{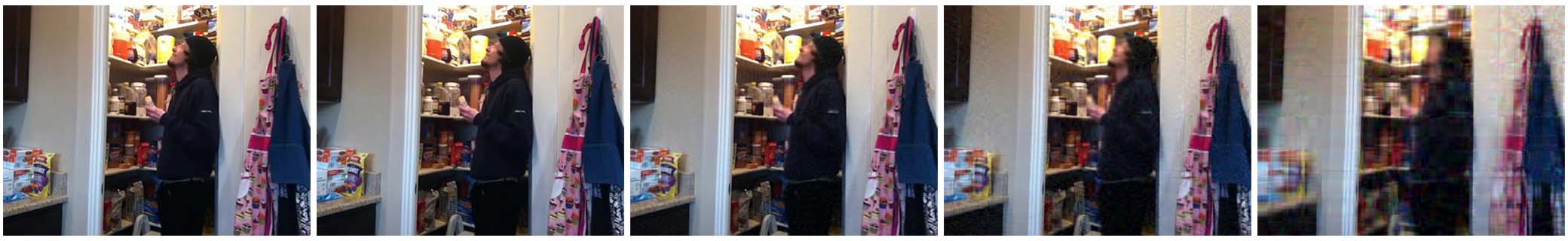}
    \caption{One   frame  from the   video  ``8JTF4.mp4''. The first frame: the original frame; second-to-last frames: frames reconstructed by truncated TS-QHOSVD   with truncation ranks   $\alpha\cdot [320,480,500]$, with $\alpha$ varing in $\{0.5, 0.2, 0.1, 0.02\}$.}
    \label{fig:compress_video_2}
\end{figure}
\begin{figure}[htbp]
    \centering
%
%
\definecolor{mycolor1}{rgb}{0.87059,0.34510,0.16863}%
\definecolor{mycolor2}{rgb}{0.09412,0.40784,0.69804}%
\definecolor{mycolor3}{rgb}{0.95294,0.63922,0.19608}%
\begin{tikzpicture}[scale=.8]

\begin{axis}[%
width=2.8in,
height=2.4in,
at={(0.5in,0.481in)},
scale only axis,
xmin=0,
xmax=0.53,
xminorticks=true,
xlabel style={font=\color{white!15!black}},
xlabel={truncated-rank ratios:},
ymin=6,
ymax=21,
yminorticks=true,
ylabel style={font=\color{white!15!black}},
ylabel={Time},
axis background/.style={fill=white},
legend style={at={(0.03,0.97)}, anchor=north west, legend cell align=left, align=left, draw=white!15!black}
]
\addplot [color=mycolor1, line width=2pt, mark=triangle, mark options={solid, mycolor1}]
  table[row sep=crcr]{%
0.02	7.34974326\\
0.04	7.82779735\\
0.06	8.05492709\\
0.08	8.27702386\\
0.1	8.72241961\\
0.12	8.91988618\\
0.14	9.27292561\\
0.16	9.62091685\\
0.18	10.20891355\\
0.2	10.33026316\\
0.22	10.50383054\\
0.24	11.03723553\\
0.26	11.12745091\\
0.28	11.40537743\\
0.3	11.66051146\\
0.32	12.06796811\\
0.34	12.05373763\\
0.36	12.24456691\\
0.38	12.891666\\
0.4	13.06776741\\
0.42	13.85525935\\
0.44	13.8764109\\
0.46	13.82420062\\
0.48	14.30894125\\
0.5	14.74123414\\
};
\addlegendentry{TS-QHOSVD}

\addplot [color=mycolor2, line width=2pt, mark=square, mark options={solid, mycolor2}]
  table[row sep=crcr]{%
0.02	7.97431845\\
0.04	8.14405902\\
0.06	8.42400298\\
0.08	8.69919148\\
0.1	9.03213041\\
0.12	9.41137613\\
0.14	9.72073225\\
0.16	10.16535594\\
0.18	10.99328342\\
0.2	11.3072469\\
0.22	11.35227513\\
0.24	12.13595081\\
0.26	12.60736254\\
0.28	12.84677153\\
0.3	13.55444023\\
0.32	13.99794712\\
0.34	14.11308005\\
0.36	14.59953854\\
0.38	14.97159145\\
0.4	15.48011933\\
0.42	16.51707859\\
0.44	16.48615022\\
0.46	16.6603878\\
0.48	17.61536819\\
0.5	18.72435788\\
};
\addlegendentry{L-QHOSVD}


\end{axis}

\begin{axis}[%
width=2.8in,
height=2.4in,
at={(4.2in,0.481in)},
scale only axis,
xmin=0,
xmax=0.53,
xminorticks=true,
xlabel style={font=\color{white!15!black}},
xlabel={truncated-rank ratios:},
ymin=0,
ymax=0.2,
yminorticks=true,
yticklabel style={/pgf/number format/fixed},
ylabel style={font=\color{white!15!black}},
ylabel={Relative Error},
axis background/.style={fill=white},
legend style={legend cell align=left, align=left, draw=white!15!black}
]
\addplot [color=mycolor1, line width=2pt, mark=triangle, mark options={solid, mycolor1}]
  table[row sep=crcr]{%
0.02	0.167576778583008\\
0.04	0.11649838491827\\
0.06	0.0927092634566148\\
0.08	0.0765987176428107\\
0.1	0.0644640796651896\\
0.12	0.0565552946045374\\
0.14	0.0498312580804556\\
0.16	0.044786610409755\\
0.18	0.0403769996671782\\
0.2	0.0364482744553833\\
0.22	0.0332846125282767\\
0.24	0.0305582270698515\\
0.26	0.0282128120188525\\
0.28	0.0260932312040474\\
0.3	0.0240562178026735\\
0.32	0.0224511041160912\\
0.34	0.0209334255458017\\
0.36	0.019589026368786\\
0.38	0.0183303537927445\\
0.4	0.017095243991087\\
0.42	0.0160999397636972\\
0.44	0.0151582603677693\\
0.46	0.0143250264659597\\
0.48	0.0135152907682172\\
0.5	0.0127031691662802\\
};
\addlegendentry{TS-QHOSVD}

\addplot [color=mycolor2, line width=2pt, mark=square, mark options={solid, mycolor2}]
  table[row sep=crcr]{%
0.02	0.165298994078464\\
0.04	0.117407572643313\\
0.06	0.0929160630397808\\
0.08	0.0765969519500924\\
0.1	0.0643603781646092\\
0.12	0.0562845064512551\\
0.14	0.0499602097524767\\
0.16	0.0445433838540971\\
0.18	0.0403064717018365\\
0.2	0.0363736140496477\\
0.22	0.0332242000646988\\
0.24	0.0304340064561443\\
0.26	0.0281747548704004\\
0.28	0.0260041858227435\\
0.3	0.0240441116071484\\
0.32	0.0224110793348442\\
0.34	0.0209334607750575\\
0.36	0.0195740445525917\\
0.38	0.0182731004064175\\
0.4	0.0170963909104689\\
0.42	0.0160765792446086\\
0.44	0.015164184818675\\
0.46	0.0143235277799047\\
0.48	0.0135063409993555\\
0.5	0.01269469754077\\
};
\addlegendentry{L-QHOSVD}


\end{axis}
\end{tikzpicture}%
    \captionsetup{aboveskip=-1pt, belowskip=-4pt}
    \caption{Comparison of elapsed time and reconstruction error for   truncated TS-HOSVD (Alg. \ref{algo:trun_ts}) and truncated L-HOSVD (Alg. \ref{algo:stl}) with varying truncation ranks  on color video ``8JTF4.mp4'' compression. }
    \label{fig:video2}
\end{figure}

We evaluate the relative reconstruction error and elapsed time of both methods, with the results presented in Fig. \ref{fig:video1} and Fig. \ref{fig:video2}. As shown in the figures,  TS-QHOSVD achieves approximately a 10\% speedup over L-QHOSVD while maintaining a comparable level of reconstruction error. The visualized reconstructed results (one frame)  by TS-QHOSVD are respectively  illustrated in Fig.~\ref{fig:compress_video_1} and Fig.~\ref{fig:compress_video_2}, which shows that the reconstructed frames well approximated the original ones for a moderate truncated rank.

\paragraph{Example 3: Scientific data: 3D Navier-Stokes Equation} 
In this example, we consider compressing the output of a computational fluid dynamics (CFD) simulation. The simulation is based on a finite element model of the unsteady 3D Navier-Stokes equations for microscopic natural convection, with applications in biological research.
The general model of Navier-Stokes equation consists of a time-dependent continuity equation for conservation of mass:
    \begin{small}
        \begin{align*}
            \frac{\partial p}{\partial t}+\frac{\partial(\rho u)}{\partial x}+\frac{\partial(\rho v)}{\partial y}+\frac{\partial(\rho w)}{\partial z}=0,
        \end{align*}
    \end{small}
    three time-dependent conservation of momentum equations:
    \begin{small}
        \begin{align*}
            &\frac{\partial(\rho u)}{\partial t}+\frac{\partial(\rho u^2)}{\partial x}+\frac{\partial(\rho uv)}{\partial y}+\frac{\partial(\rho uw)}{\partial z}=-\frac{\partial\rho}{\partial x}+\frac{1}{Re_{r}}[\frac{\partial\tau_{xx}}{\partial x}+\frac{\partial\tau_{xy}}{\partial y}+\frac{\partial\tau_{xz}}{\partial z}] \\
            &\frac{\partial(\rho v)}{\partial t}+\frac{\partial(\rho uv)}{\partial x}+\frac{\partial(\rho v^2)}{\partial y}+\frac{\partial(\rho vw)}{\partial z}=-\frac{\partial\rho}{\partial y}+\frac1{Re_r}[\frac{\partial\tau_{xy}}{\partial x}+\frac{\partial\tau_{yy}}{\partial y}+\frac{\partial\tau_{yz}}{\partial z}] \\
            &\frac{\partial(\rho w)}{\partial t}+\frac{\partial(\rho uw)}{\partial x}+\frac{\partial(\rho vw)}{\partial y}+\frac{\partial(\rho w^2)}{\partial z}=-\frac{\partial\rho}{\partial z}+\frac1{Re_r}[\frac{\partial\tau_{xz}}{\partial x}+\frac{\partial\tau_{yz}}{\partial y}+\frac{\partial\tau_{zz}}{\partial z}],
        \end{align*}
    \end{small}
    and a time-dependent conservation of energy equation:
    \begin{small}
        \begin{equation*}\begin{aligned}
            &\frac{\partial(E_t)}{\partial t}+\frac{\partial(uE_t)}{\partial x}+\frac{\partial(vE_t)}{\partial y}+\frac{\partial(wE_t)}{\partial z}=-\frac{\partial(up)}{\partial x}-\frac{\partial(vp)}{\partial y}-\frac{\partial(wp)}{\partial z}-\frac1{Re_rPr_r}[\frac{\partial q_x}{\partial x}+\frac{\partial q_y}{\partial y}+\frac{\partial q_z}{\partial z}]\\&+\frac1{Re_r}[\frac\partial{\partial x}(u\tau_{xx}+v\tau_{xy}+w\tau_{xz})+\frac\partial{\partial y}(u\tau_{xy}+v\tau_{yy}+w\tau_{yz})+\frac\partial{\partial z}(u\tau_{xz}+v\tau_{yz}+w\tau_{zz})].
        \end{aligned}\end{equation*} 
    \end{small}
    The data was collected using the QuickerSim CFD Toolbox for MATLAB.
    
    \begin{figure}[htbp]
    \centering
%
%
\definecolor{mycolor1}{rgb}{0.87059,0.34510,0.16863}%
\definecolor{mycolor2}{rgb}{0.09412,0.40784,0.69804}%
\definecolor{mycolor3}{rgb}{0.00392,0.54118,0.40392}%
\begin{tikzpicture}[scale=.8]

\begin{axis}[%
width=2.8in,
height=2.4in,
at={(0.5in,0.481in)},
scale only axis,
xmin=0,
xmax=53,
xminorticks=true,
xlabel style={font=\color{white!15!black}},
xlabel={r:},
ymin=100,
ymax=350,
yminorticks=true,
ylabel style={font=\color{white!15!black}},
ylabel={Time},
axis background/.style={fill=white},
legend style={at={(0.03,0.97)}, anchor=north west, legend cell align=left, align=left, draw=white!15!black}
]
\addplot [color=mycolor1, line width=2pt, mark=triangle, mark options={solid, mycolor1}]
  table[row sep=crcr]{%
5	123.5001548\\
10	124.5012992\\
15	131.5835873\\
20	131.7268966\\
25	136.6342892\\
30	135.7141828\\
35	140.3053548\\
40	147.5479492\\
45	153.3746598\\
50	160.2559764\\
};
\addlegendentry{TS-QHOSVD}

\addplot [color=mycolor2, line width=2pt, mark=square, mark options={solid, mycolor2}]
  table[row sep=crcr]{%
5	240.2708538\\
10	237.2820208\\
15	236.6911586\\
20	240.4245538\\
25	251.2777863\\
30	245.1365692\\
35	270.8569553\\
40	272.7314581\\
45	303.6184206\\
50	338.8268441\\
};
\addlegendentry{L-QHOSVD}

\addplot [color=mycolor3, line width=2pt, mark=diamond, mark options={solid, mycolor3}]
  table[row sep=crcr]{%
5	157.7074351\\
10	172.0048598\\
15	181.2985025\\
20	186.5099416\\
25	192.8096128\\
30	202.0945492\\
35	208.7209594\\
40	219.6069516\\
45	223.7771009\\
50	227.5338393\\
};
\addlegendentry{seq-HOSVD}

\end{axis}

\begin{axis}[%
width=2.8in,
height=2.4in,
at={(4.2in,0.481in)},
scale only axis,
xmin=0,
xmax=53,
xminorticks=true,
xlabel style={font=\color{white!15!black}},
xlabel={r:},
ymin=0,
ymax=0.022,
yminorticks=true,
ylabel style={font=\color{white!15!black}},
ylabel={Relative Error},
axis background/.style={fill=white},
legend style={legend cell align=left, align=left, draw=white!15!black}
]
\addplot [color=mycolor1, line width=2pt, mark=triangle, mark options={solid, mycolor1}]
  table[row sep=crcr]{%
5	0.0104610929639913\\
10	0.00537082785054008\\
15	0.00303121323724334\\
20	0.00192208397092344\\
25	0.00128232058633831\\
30	0.000916436948291146\\
35	0.000685212992773732\\
40	0.000533965187000007\\
45	0.000433948762666497\\
50	0.000359931188363598\\
};
\addlegendentry{TS-QHOSVD}

\addplot [color=mycolor2, line width=2pt, mark=square, mark options={solid, mycolor2}]
  table[row sep=crcr]{%
5	0.0117623911960026\\
10	0.00616063443974177\\
15	0.00384185088106806\\
20	0.00266322365071924\\
25	0.001819293376632\\
30	0.00139699557904561\\
35	0.00119359680233259\\
40	0.000945620529355834\\
45	0.000790083934704332\\
50	0.000692466020585201\\
};
\addlegendentry{L-QHOSVD}

\addplot [color=mycolor3, line width=2pt, mark=diamond, mark options={solid, mycolor3}]
  table[row sep=crcr]{%
5	0.0201900878861834\\
10	0.0107563060185293\\
15	0.0070617052372834\\
20	0.0049148284756134\\
25	0.00377568158694744\\
30	0.00296973501698927\\
35	0.00237179188042504\\
40	0.00196069035334028\\
45	0.00165003523247044\\
50	0.00140133926300692\\
};
\addlegendentry{seq-HOSVD}

\end{axis}
\end{tikzpicture}%
    \captionsetup{aboveskip=-1pt, belowskip=-4pt}
    \caption{Comparison of elapsed time and reconstruction error for   truncated TS-HOSVD (Alg. \ref{algo:trun_ts}), L-HOSVD (Alg. \ref{algo:stl}), and sequential HOSVD (Alg. \ref{algo:sequential_hosvd}) on 3D NS equation data compression.}
    \label{fig:cfd_compress}
    \end{figure}

    In this example, the simulation domain contains $20914$ spatial nodes, each with a time-dependent 3D velocity vector. The velocity field is sampled at $200$ time instants. Additionally, the kinematic viscosity parameter $k$ is varied from \(2.01~{\rm mm}^2/{\rm s}\) to \(4.00~{\rm mm}^2/{\rm s}\) in increments of \(0.01\), yielding $200$ evenly spaced values. These correspond to the third mode of the resulting quaternion tensor, capturing the effect of viscosity variations on the velocity field, yielding a pure quaternion tensor of size \(20914 \times 200 \times 200\) ($18.69~{\rm GB}$ in memory, $4.86$GB in disk), where each frontal slice corresponds to a distinct viscosity value and encodes the full temporal evolution of the velocity field. 
    Each data entry $t_{p,t,k}$ in the tensor corresponds to a pure quaternion that encodes the 3D velocity vector:
    \[
    t_{p,t,k} = v_x(p,t,k){\bm i} + v_y(p,t,k){\bm j} + v_z(p,t,k){\bm k},
    \]
    where \(v = (v_x, v_y, v_z)\) is the velocity vector field, indexed by spatial node \(p\), time \(t\), and kinematic viscosity parameter \(k\).
 
    We apply truncated TS-QHOSVD (Alg. \ref{algo:trun_ts})  and truncated L-QHOSVD (Alg. \ref{algo:stl}) with different truncated ranks $[r,r,r]$  for $r = 5:5:50$. We also consider the sequential HOSVD  (Alg. \ref{algo:sequential_hosvd}) with truncated rank $[r,r,r,3]$, by treating the quaternion tensor as a fourth-order real tensor.   The results are shown in Fig~\ref{fig:cfd_compress}, from which we observe  that   TS-QHOSVD achieves a $75.0\% $ reduction in reconstruction error compared to sequential HOSVD, under the same truncation conditions, and TS-QHOSVD slightly outperforms L-QHOSVD. In addition, TS-QHOSVD demonstrates   speed advantages over the compared methods. At a truncation rank of \([30,30,30]\), the original $4.86$ GB dataset is compressed by TS-QHOSVD to $12.0$ MB(compression ratio is $99.76\%$), while retaining approximately $99.91\%$ of the relative accuracy. To  visualize it, we illustrate the shear rates computed from velocity field in Fig.~\ref{cfd}, where the first row is the original share rates, while the second row is the reconstructed ones. We observe that the reconstructed ones  are closely match the real data.  

    \begin{figure}[htbp]
        \centering
        \includegraphics[width=\linewidth]{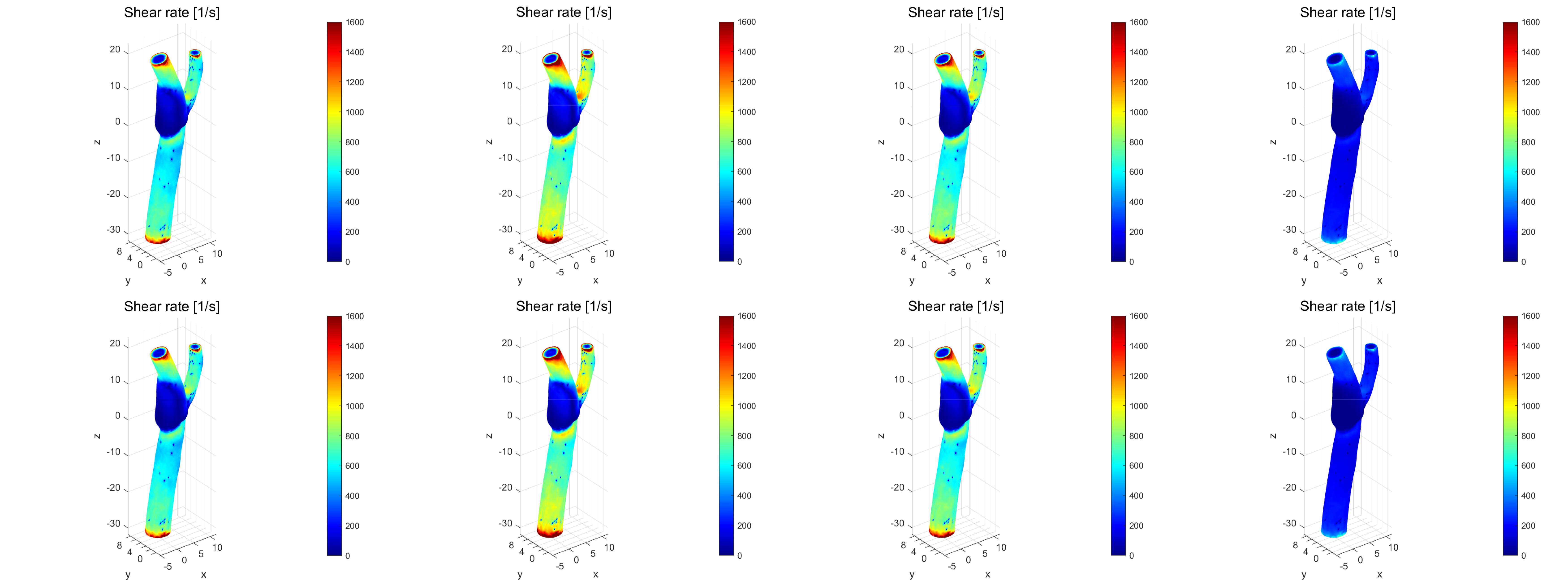}
        \captionsetup{aboveskip=-1pt, belowskip=-4pt}
        \caption{Reconstructed data results using truncated TS-HOSVD. The first row shows the original share rates, while the second row shows the reconstructed ones. The first to fourth columns correspond to viscosity values of 2.5, 3, 3.5, and 4, with the corresponding steps being 7th, 11th, 17th, and 23th, respectively.}
        \label{cfd}
    \end{figure}




 \paragraph{Example 4: Scientific data: Lorenz system}

In this example, we consider compressing data generated from the Lorenz attractor. The Lorenz system is governed by the following set of ordinary differential equations:

$$
\frac{dx}{dt} = \alpha(y - x), \quad \frac{dy}{dt} = x(\rho - z) - y, \quad \frac{dz}{dt} = xy - \beta z,
$$

where $\alpha$, $\beta$, and $\rho$ are positive constants. To ensure that the system exhibits chaotic behavior, we set the parameters as $\alpha = 10$, $\beta = \frac{8}{3}$, and $\rho = 28$. The system trajectories are numerically computed using MATLAB's built-in ODE solver \texttt{ode45}.

We generate $10^8$ trajectory points, which are reshaped into a five-dimensional tensor of size $100 \times 100 \times 100 \times 100 \times 3$, representing spatial-temporal evolution and vector components $(x, y, z)$. This tensor is then treated as a four-dimensional pure quaternion tensor $\zl T \in \mathbb{H}^{100 \times 100 \times 100 \times 100}$ ($2.23{\rm GB}$ in memory, $2.15~{\rm GB}$ in disk).

We compare   truncated TS-QHOSVD (Alg. \ref{algo:trun_ts}), L-QHOSVD (Alg. \ref{algo:stl}), and sequential HOSVD (Alg. \ref{algo:sequential_hosvd}). 
 The reconstruction errors and elapsed times are summarized in Fig.~\ref{fig:lorentz}. The results clearly demonstrate that     TS-QHOSVD and L-QHOSVD  both outperform the traditional sequential HOSVD in terms of reconstruction accuracy. Moreover, TS-QHOSVD again achieves the shortest elapsed time among all methods. 
\begin{figure}[htbp]
    \centering
%
%
\definecolor{mycolor1}{rgb}{0.87059,0.34510,0.16863}%
\definecolor{mycolor2}{rgb}{0.09412,0.40784,0.69804}%
\definecolor{mycolor3}{rgb}{0.95294,0.63922,0.19608}%
\definecolor{mycolor4}{rgb}{0.00392,0.54118,0.40392}%
\begin{tikzpicture}[scale=.8]

\begin{axis}[%
width=2.8in,
height=2.4in,
at={(0.5in,0.481in)},
scale only axis,
xmin=40,
xmax=80,
xminorticks=true,
xlabel style={font=\color{white!15!black}},
xlabel={r:},
ymin=6,
ymax=20,
yminorticks=true,
ylabel style={font=\color{white!15!black}},
ylabel={Time},
axis background/.style={fill=white},
legend style={at={(0.03,0.97)}, anchor=north west,legend cell align=left, align=left, draw=white!15!black, fill = white, fill opacity=0.8, draw=none}
]
\addplot [color=mycolor1, line width=2pt, mark=triangle, mark options={solid, mycolor1}]
  table[row sep=crcr]{%
40	6.74105624\\
42	7.05236596\\
44	7.09035556\\
46	7.54703324\\
48	7.69694075\\
50	8.02705808\\
52	7.93903266\\
54	8.44770231\\
56	8.46878503\\
58	8.90133591\\
60	8.89972287\\
62	9.35093071\\
64	10.41126974\\
66	9.9174971\\
68	9.98891246\\
70	10.52181258\\
72	10.57315265\\
74	11.05780579\\
76	11.0237952\\
78	11.79643091\\
80	12.00783473\\
};
\addlegendentry{TS-QHOSVD}

\addplot [color=mycolor2, line width=2pt, mark=square, mark options={solid, mycolor2}]
  table[row sep=crcr]{%
40	8.21926824\\
42	8.60923064\\
44	8.75886242\\
46	9.43794898\\
48	9.59776122\\
50	10.20132748\\
52	10.25955483\\
54	11.15331459\\
56	11.24720608\\
58	11.92866372\\
60	12.12722163\\
62	12.74042319\\
64	13.87870518\\
66	13.79362809\\
68	14.10316968\\
70	14.8881188\\
72	15.43092724\\
74	16.27224354\\
76	16.52182763\\
78	17.64915931\\
80	18.45959618\\
};
\addlegendentry{L-QHOSVD}


\addplot [color=mycolor4, line width=2pt, mark=diamond, mark options={solid, mycolor4}]
  table[row sep=crcr]{%
40	8.34663568\\
42	8.60060846\\
44	8.8166648\\
46	9.0766453\\
48	9.32238348\\
50	9.54755535\\
52	9.8729998\\
54	10.15105371\\
56	10.50262334\\
58	10.68811223\\
60	11.06299088\\
62	11.38131759\\
64	12.48407431\\
66	12.08073468\\
68	12.4309673\\
70	12.79915256\\
72	13.30349648\\
74	13.57235274\\
76	14.05703701\\
78	14.43568435\\
80	14.99820257\\
};
\addlegendentry{seq-HOSVD}

\end{axis}

\begin{axis}[%
width=2.8in,
height=2.4in,
at={(4.2in,0.481in)},
scale only axis,
xmin=40,
xmax=80,
xminorticks=true,
xlabel style={font=\color{white!15!black}},
xlabel={r:},
ymin=0.3,
ymax=0.52,
yminorticks=true,
ylabel style={font=\color{white!15!black}},
ylabel={Relative Error},
axis background/.style={fill=white},
legend style={legend cell align=left, align=left, draw=white!15!black, fill = white, fill opacity=0.8, draw=none}
]
\addplot [color=mycolor1, line width=2pt, mark=triangle, mark options={solid, mycolor1}]
  table[row sep=crcr]{%
40	0.494346543268192\\
42	0.490201146343549\\
44	0.485745885550417\\
46	0.480950126156879\\
48	0.475849917031457\\
50	0.470379616793584\\
52	0.464472469652644\\
54	0.458288791165586\\
56	0.45179274146189\\
58	0.444798925890639\\
60	0.437366695289192\\
62	0.42937346780881\\
64	0.420982680066767\\
66	0.411962006971016\\
68	0.402390734972645\\
70	0.391987752482492\\
72	0.38119602413075\\
74	0.369602049942446\\
76	0.357436515873216\\
78	0.344209265626778\\
80	0.330067441865736\\
};
\addlegendentry{TS-QHOSVD}

\addplot [color=mycolor2, line width=2pt, mark=square, mark options={solid, mycolor2}]
  table[row sep=crcr]{%
40	0.493138060542658\\
42	0.488925181005352\\
44	0.484351664958237\\
46	0.479521411513334\\
48	0.474371533735868\\
50	0.468916337037214\\
52	0.463082206684902\\
54	0.456784249838404\\
56	0.450085844597077\\
58	0.443124352055222\\
60	0.435501449480928\\
62	0.427461816847321\\
64	0.418959949450409\\
66	0.409865358382291\\
68	0.400266218980256\\
70	0.390071076232039\\
72	0.379220021336877\\
74	0.367554311595426\\
76	0.355366268422847\\
78	0.342088295260373\\
80	0.327859057759052\\
};
\addlegendentry{L-QHOSVD}


\addplot [color=mycolor4, line width=2pt, mark=diamond, mark options={solid, mycolor4}]
  table[row sep=crcr]{%
40	0.498483490538544\\
42	0.49474392014438\\
44	0.490738914348041\\
46	0.486414290052789\\
48	0.481763687068332\\
50	0.476808774688655\\
52	0.471495592789771\\
54	0.465801043911422\\
56	0.459711320956562\\
58	0.453229411502076\\
60	0.446305479857063\\
62	0.43889089399023\\
64	0.430987843920556\\
66	0.42252874153035\\
68	0.413522653087333\\
70	0.403863315408132\\
72	0.393556405298044\\
74	0.382483899181832\\
76	0.370589273073143\\
78	0.357787137038041\\
80	0.344002130659407\\
};
\addlegendentry{seq-HOSVD}

\end{axis}
\end{tikzpicture}%
    \captionsetup{aboveskip=-1pt, belowskip=-4pt}
    \caption{Comparison of elapsed time and reconstruction error for   truncated TS-HOSVD (Alg. \ref{algo:trun_ts}), L-HOSVD (Alg. \ref{algo:stl}), and sequential HOSVD (Alg. \ref{algo:sequential_hosvd}) on  on the Lorenz system compression.}
    \label{fig:lorentz}
\end{figure}

\paragraph{Example 5: Video denoising}
In this example, we compare rank truncated TS-QHOSVD (Alg. \ref{algo:trun_ts}) with the denoising strategy of hard-thresholding   L-QHOSVD \cite[Sect. 5.2]{miao2023quat} on color video denoising. \cite{miao2023quat}'s strategy is as follows: first, it computes the full L-QHOSVD of the tensor $\zl{T}$ (see Sect. \ref{sec:lqhosvd}). Then  it shrinks the core tensor $\zl{S}$ to obtain the new core $\hat{\zl{S}}$, with $
\hat{\zl S}_{ijk} = \begin{cases}
    {\zl S}_{ijk}, & |{\zl S}_{ijk}| > \tau \\
    0, & \text{otherwise}
\end{cases}
$;  here  $\tau= \eta \sigma \sqrt{\log(2I_1I_2I_3)}$,  where $\sigma$ denotes the noise level and $\eta$ is a tuning parameter.  Finally, the denoised tensor is given by $\hat{\zl{T}} = \jz{U}_1 \times^L_1 \jz{U}_2 \times^L_2 \jz{U}_3 \times^L_3 \hat{\zl{S}}$. Note that $\hat{\zl{S}}$ has the same size as $\zl{S}$. 
 
We use the test video ``akiyo'', a YUV video downloaded from (\url{http://trace.eas.asu.edu/yuv/index.html}). The video contains $300$ color frames, each of size $288 \times 352$. By treating the RGB channels as the imaginary components of a quaternion, we construct a pure quaternion tensor $\zl{T} \in \quat{288 \times 352 \times 300}$.   Gaussian noise ($\sigma = 30$) was added to the video.

As the two methods use different strategies, to   ensure a fair comparison, we control the number of preserved (non-zero) elements in the core tensor and fator matrices to be approximately the same, namely, to compare the denoising ability at the same level of storage cost. Each experiment is repeated 10 times, and the results are averaged. we report reconstructed error, PSNR (Peak Signal-to-Noise Ratio), SSIM (Structural Similarity Index Measure), number of preserved elements, and elapsed time   in Table~\ref{akiyo denoising table}. The visual denoising results (one frame) are presented in Fig.~\ref{fig:denoisy}.

\begin{figure}[htbp]
    \centering
    \includegraphics[width=\linewidth]{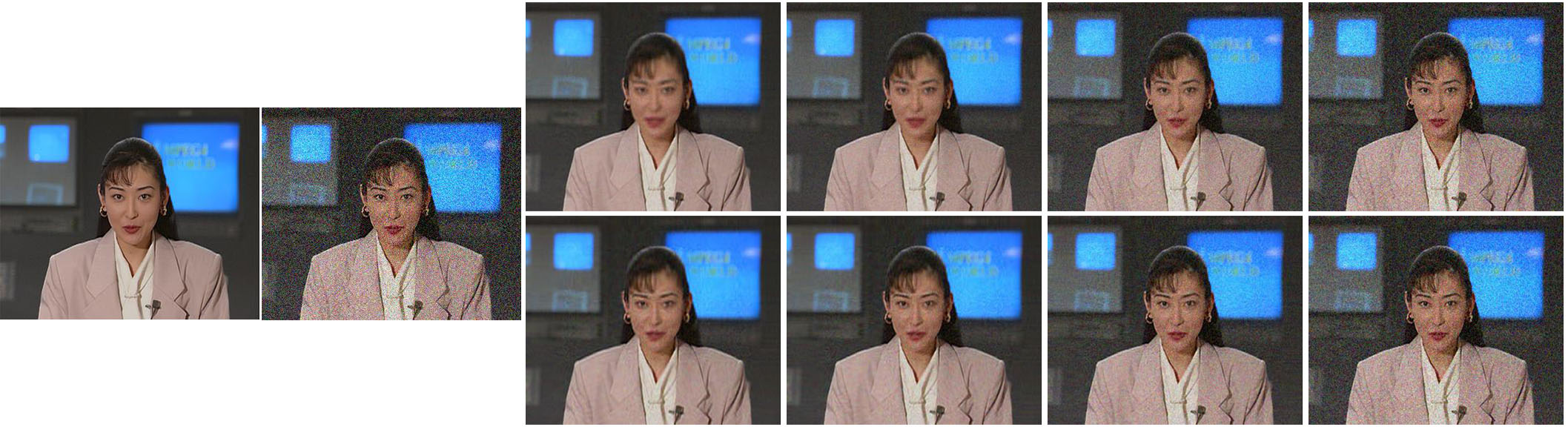}
    \captionsetup{aboveskip=-1pt, belowskip=-4pt}
    \caption{One frame of the denoising results.  The first is the original clean frame, while the second   displays the corresponding noisy frame.  Remaining figures in the first  row: hard-thresholded   L-QHOSVD with $\eta$ values of $0.0025$, $ 0.002$, $0.0015$, and $0.001$. Bottom row: rank-truncated TS-QHOSVD with corresponding truncation ranks to achieve similar preserved elements.  }
    \label{fig:denoisy}
\end{figure}

\begin{table}[htbp]
    \centering
    \captionsetup{aboveskip=-1pt, belowskip=-4pt}
    \caption{Denoising results on color video ``akiyo''; ``num.'' denotes the number of preserved elements.}
    \label{akiyo denoising table}
\resizebox{\textwidth}{!}{
    \begin{mytabular1}{cccccccccccc}
        \toprule[1pt]
        \multicolumn{6}{c}{hard-thresholding L-QHOSVD \cite[Sect. 5.2]{miao2023quat}} & \multicolumn{6}{c}{truncated TS-QHOSVD (Alg. \ref{algo:trun_ts})}  \\
        \cline{1-12}
        $\eta$ & num. & time & psnr & ssim & re.err & truncated size & num. & time & psnr & ssim & re.err\\           
        \midrule[1pt]

        $\eta = 0.0025$ & 441564.2 & 14.05 & $\bm{32.93}$ & $\bm{0.85}$ & $\bm{0.0497}$ & [70,75,70] & $\bm{435060}$ & $\bm{4.31}$ &  32.07 & 0.85 & 0.0548 \\

        $\eta = 0.002$ & 1077309.2 & 14.09 & 29.46 & 0.63 & 0.0741 & [95,100,95] & $\bm{993560}$ & $\bm{4.71}$ & $\bm{31.58}$ & $\bm{0.79}$ & $\bm{0.0580}$ \\

        $\eta = 0.0015$ & 5293946.2 & 14.17 & 23.54 & 0.33 & 0.1464 & [170,175,170] &  $\bm{5219060}$ & $\bm{6.30}$ & $\bm{26.37}$ & $\bm{0.48}$ & $\bm{0.1057}$ \\

        $\eta = 0.001$ & 17580931.2 & 14.27 & 19.67 & 0.21 & 0.290 & [230,300,250] & $\bm{174968400}$ & $\bm{9.16}$ & $\bm{21.21}$ & $\bm{0.25}$ & $\bm{0.1914}$ \\ 
        \bottomrule[1pt]              
    \end{mytabular1}}\\
    Boldface means that the best among all the   methods
\end{table}

From the results, we observe that under similar storage situations, truncated TS-QHOSVD outperforms     L-QHOSVD in terms of speed, which still stems from its parallelizable nature. As more components are retained, TS-QHOSVD demonstrates improved denoising performance. Notably, even when the number of preserved elements of TS-QHOSVD are less than that of L-QHOSVD (see columns $2$ and $8$ in Table \ref{akiyo denoising table}), it still yields better results  in robustness and quality (in terms of PSNR, SSIM, and relative error). This highlights TS-QHOSVD's resilience to rank parameter tuning and its potential in practical denoising tasks.

\section{Conclusions}\label{conclusions} 
Quaternion HOSVD was recently proposed for quaternion tensors \cite{miao2023quat}, in the fashion of the sequential HOSVD \cite{vannieuwenhoven2012new}. To take advantage of model computing architechture,  this work introduces a two-sided quaternion HOSVD (TS-QHOSVD) that can be parallelized on two processors. We show that TS-QHOSVD      preserves the   HOSVD ordering property,  inherits the orthogonality property     at the first and the last modes,  and   satisfies the weak orthogonality     at all modes.  The rank truncated TS-QHOSVD is also developed and its error bound, which is controlled by the tail energy, has been established. We have applied TS-QHOSVD to   color video denoising   as well as scientific data compression (3D Navier-Stokes equation and Lorentz system) tasks to demonstrate its efficacy.

{\scriptsize\paragraph*{Acknowledgement}   This work was funded by the    National Natural Science Foundation of China No. 12171105, the special foundation for Guangxi Ba Gui Scholars No. GXR-6BG2424008, and Guangxi Science and Technology Plan Project for Research Bases and Talents of China  No.  AD22080047.	}

\noindent{\footnotesize {\bf Conflict of interest} The authors have no conflict of interest to declare that are relevant to the content of this
article.}

\noindent {\footnotesize {\bf Ethical approval} The datasets and algorithms generated during the current study are available from the corresponding author on reasonable request.}

\bibliographystyle{plain}
\bibliography{ref2}

\begin{thebibliography}{10}

\bibitem{chen2022color}
J.~Chen and M.~K. Ng.
\newblock Color image inpainting via robust pure quaternion matrix completion: Error bound and weighted loss.
\newblock {\em SIAM J. Imag. Sci.}, 15(3):1469--1498, 2022.

\bibitem{de2008decompositions}
L.~De~Lathauwer.
\newblock Decompositions of a higher-order tensor in block terms—part {I}: Lemmas for partitioned matrices.
\newblock {\em SIAM J. Matrix Anal. Appl.}, 30(3):1022--1032, 2008.

\bibitem{hosvd}
L.~De~Lathauwer, B.~De~Moor, and J.~Vandewalle.
\newblock A multilinear singular value decomposition.
\newblock {\em SIAM J. Matrix Anal. Appl.}, 21(4):1253--1278, 2000.

\bibitem{flamant2024multilinear}
J.~Flamant, X.~Luciani, S.~Miron, and Y.~Zniyed.
\newblock Multilinear analysis of quaternion arrays: theory and computation.
\newblock {\em arXiv preprint arXiv:2412.05409}, 2024.

\bibitem{he2023eigenvalues}
Z.-H. He, X.-X. Wang, and Y.-F. Zhao.
\newblock Eigenvalues of quaternion tensors with applications to color video processing.
\newblock {\em J. Sci. Comput.}, 94(1):1, 2023.

\bibitem{osimone2023low}
O.~Imhogiemhe, J.~Flamant, X.~Luciani, Y.~Zniyed, and S.~Miron.
\newblock Low-rank tensor decompositions for quaternion multiway arrays.
\newblock In {\em ICASSP 2023-2023 IEEE International Conference on Acoustics, Speech and Signal Processing (ICASSP)}, pages 1--5. IEEE, 2023.

\bibitem{jia2022non}
Z.~Jia, Q.~Jin, M.~K. Ng, and X.-L. Zhao.
\newblock Non-local robust quaternion matrix completion for large-scale color image and video inpainting.
\newblock {\em IEEE Trans. Image Process.}, 31:3868--3883, 2022.

\bibitem{jia2021structure}
Z.~Jia and M.~K. Ng.
\newblock Structure preserving quaternion generalized minimal residual method.
\newblock {\em SIAM J. Matrix Anal. Appl.}, 42(2):616--634, 2021.

\bibitem{jia2019lanczos}
Z.~Jia, M.~K. Ng, and G.-J. Song.
\newblock Lanczos method for large-scale quaternion singular value decomposition.
\newblock {\em Numer. Algorithms}, 82:699--717, 2019.

\bibitem{jia2019robust}
Z.~Jia, M.~K. Ng, and G.-J. Song.
\newblock Robust quaternion matrix completion with applications to image inpainting.
\newblock {\em Numer. Linear Algebra Appl.}, 26(4):e2245, 2019.

\bibitem{jia2018new}
Z.~Jia, M.~Wei, M.-X. Zhao, and Y.~Chen.
\newblock A new real structure-preserving quaternion qr algorithm.
\newblock {\em J. Comput. Applied Math.}, 343:26--48, 2018.

\bibitem{kilmer2011factorization}
M.~E. Kilmer and C.~D. Martin.
\newblock Factorization strategies for third-order tensors.
\newblock {\em Linear Algebra Appl.}, 435(3):641--658, 2011.

\bibitem{kolda2009tensor}
T.~G. Kolda and B.~W. Bader.
\newblock Tensor decompositions and applications.
\newblock {\em SIAM Rev.}, 51(3):455--500, 2009.

\bibitem{miao2020quaternion}
J.~Miao and K.~I. Kou.
\newblock Quaternion-based bilinear factor matrix norm minimization for color image inpainting.
\newblock {\em IEEE Trans. Signal Process.}, 68:5617--5631, 2020.

\bibitem{miao2023qt-svd}
J.~Miao and K.~I. Kou.
\newblock Quaternion tensor singular value decomposition using a flexible transform-based approach.
\newblock {\em Signal Process.}, 206:108910, 2023.

\bibitem{miao2023quat}
J.~Miao, K.~I. Kou, D.~Cheng, and W.~Liu.
\newblock Quaternion higher-order singular value decomposition and its applications in color image processing.
\newblock {\em Inf. Fusion}, 92:139--153, 2023.

\bibitem{miao2020low}
J.~Miao, K.~I. Kou, and W.~Liu.
\newblock Low-rank quaternion tensor completion for recovering color videos and images.
\newblock {\em Pattern Recognit.}, 107:107505, 2020.

\bibitem{miao2022quat}
J.~Miao, K.~I. Kou, L.~Yang, and D.~Cheng.
\newblock Quaternion tensor train rank minimization with sparse regularization in a transformed domain for quaternion tensor completion, 2022.

\bibitem{miron2023quaternions}
S.~Miron, J.~Flamant, N.~Le~Bihan, P.~Chainais, and D.~Brie.
\newblock Quaternions in signal and image processing: A comprehensive and objective overview.
\newblock {\em IEEE Signal Process. Mag.}, 40(6):26--40, 2023.

\bibitem{oseledets2011tensor}
I.~V. Oseledets.
\newblock Tensor-train decomposition.
\newblock {\em SIAM J. Sci. Comput.}, 33(5):2295--2317, 2011.

\bibitem{qi2022quaternion}
L.~Qi, Z.~Luo, Q.-W. Wang, and X.~Zhang.
\newblock Quaternion matrix optimization: Motivation and analysis.
\newblock {\em J. Optim. Theory Appl.}, 193(1-3):621--648, 2022.

\bibitem{qin2022sigular}
Z.~Qin, Z.~Ming, and L.~Zhang.
\newblock Singular value decomposition of third order quaternion tensors.
\newblock {\em Applied Math. Lett.}, 123:107597, 2022.

\bibitem{rodman2014topics}
L.~Rodman.
\newblock {\em Topics in Quaternion Linear Algebra}.
\newblock Princeton University Press, 2014.

\bibitem{schulz2014using}
D.~Schulz and R.~S. Thoma.
\newblock Using quaternion-valued linear algebra.
\newblock {\em arXiv preprint, arXiv:1311.7488}, 2014.

\bibitem{vannieuwenhoven2012new}
N.~Vannieuwenhoven, R.~Vandebril, and K.~Meerbergen.
\newblock A new truncation strategy for the higher-order singular value decomposition.
\newblock {\em SIAM J. Sci. Comput.}, 34(2):A1027--A1052, 2012.

\bibitem{wei2018quaternion}
M.~Wei, Y.~Li, F.~Zhang, and J.~Zhao.
\newblock {\em Quaternion matrix computations}.
\newblock Nova Science Publishers, Incorporated, 2018.

\bibitem{xu2015optimization}
D.~Xu, Y.~Xia, and D.~P. Mandic.
\newblock Optimization in quaternion dynamic systems: Gradient, hessian, and learning algorithms.
\newblock {\em IEEE Trans. Neural Netw. Learn. Syst.}, 27(2):249--261, 2015.

\bibitem{zhang1997quaternions}
F.~Zhang.
\newblock Quaternions and matrices of quaternions.
\newblock {\em Linear Algebra Appl.}, 251:21--57, 1997.

\bibitem{zhao2016tensor}
Q.~Zhao, G.~Zhou, S.~Xie, L.~Zhang, and A.~Cichocki.
\newblock Tensor ring decomposition.
\newblock {\em arXiv preprint arXiv:1606.05535}, 2016.

\bibitem{zheng2023approximation}
M.-M. Zheng and G.~Ni.
\newblock Approximation strategy based on the t-product for third-order quaternion tensors with application to color video compression.
\newblock {\em Applied Math. Lett.}, 140:108587, 2023.

\bibitem{zou2016quat}
C.~Zou, K.~I. Kou, and Y.~Wang.
\newblock Quaternion collaborative and sparse representation with application to color face recognition.
\newblock {\em IEEE Trans. Image Process.}, 25(7):3287--3302, 2016.

\bibitem{zou2021quat}
C.~Zou, K.~I. Kou, Y.~Wang, and Y.-Y. Tang.
\newblock Quaternion block sparse representation for signal recovery and classification.
\newblock {\em Signal Process.}, 179:107849, 2021.

\end{thebibliography}

\appendix
\renewcommand{\theequation}{A\arabic{equation}}
\setcounter{equation}{0}

\small
\section{Proof of Lemma~\ref{lem:mode-k_unfolding}} \label{sec:proof_lem_k_mode_right_unfolding}
We only present the proof of $k$-mode right unfolding, while that for $k$-mode left unfolding is analogous.

    Let $\zl{T}\in \quat{I_1\times \cdots \times I_{N}}$ and $\jz{V}_m\in \quat{I_m \times J_m}(m=1,2, \dots,k)$. Denote
    \[\jz{G}:=\jz{V}_N \otimes \cdots \otimes \jz{V}_{k} \otimes \jz{I}^{(k-1)} \otimes \cdots \otimes  \jz{I}^{(j+1)} \otimes \jz{I}^{(j-1)} \otimes \cdots \otimes \jz{I}^{(1)}.\]
    Our goal is to show that every entry $[\zl{T} \times^R_N \jz{V}_N \cdots \times^R_k \jz{V}_k]_{i_1\dots i_{k-1} j_k\dots j_N}$ is the same as the entry
    $\jz{T}^L_{[j]}\jz{G}(i_j,r)$.
    From \refdef{lproduct}, we have 
    \[[\zl{T} \times^R_N \jz{V}_N \cdots \times^R_k \jz{V}_k]_{i_1\dots i_{k-1} j_k\dots j_N}
    =\sum_{i_1=1}^{I_1}\cdots \sum_{i_N=1}^{I_N} t_{i_1i_2\dots i_N}v_{i_Nj_N}\cdots v_{i_kj_k}.\]
    On the other hand, we have 
    \begin{equation} \label{eq:element}
        \begin{split}
             \jz{T}^L_{[j]}\jz{G}(i_j,r)  
            =  \jz{T}^L_{[j]}(i_j,:)\jz{G}(:,r)  
            =  \sum_{p=1}^{P} \jz{T}^L_{[j]}(i_j,p)\jz{G}(p,r),
        \end{split}
    \end{equation}
    where
    \begin{align*}
        & P=I_N \cdots I_{j+1} I_{j-1} \cdots I_1;\quad p=1+\sum_{\substack{l=1 \\l\neq j}}^{N}\nolimits(i_l-1) P_l\quad \text{with}\quad 
        P_l=\prod \nolimits_{\substack{m=1 \\ m\neq j}}^{l-1}I_m. \\
        & r= 1+ \sum_{l=1}^{j-1} (i_l -1)R_l + \sum_{l=j+1}^{N}(j_l -1)R_l \quad \text{with} \quad R_l=\begin{cases}
            \prod \nolimits_{\substack{d=1}}^{l-1}I_d, & l< j\\
            \prod \nolimits_{\substack{d=1}}^{j-1}I_d, & l=j+1 \\
            (\prod \nolimits_{\substack{d=1}}^{j-1}I_d)(\prod \nolimits_{\substack{d=j+1}}^{l-1}J_d), & l>j+1.
        \end{cases} \\
    \end{align*}
    From the definition of Kronecker product we have 
    \[\jz{G}(p,r)=\jz{V}_N(i_N,j_N)\cdots \jz{V}_k(i_k,j_k)\jz{I}^{(k-1)}(i_{k-1},i_{k-1})\cdots \jz{I}^{(1)}(i_1,i_1)=\jz{V}_N(i_N,j_N)\cdots \jz{V}_k(i_k,j_k).\]
    Note that $\jz{T}^L_{[j]}(i_j,p)=t_{i_1i_2\dots i_N}$ and $\jz{V}(i_k,j_k)=v_{i_kj_k}$.
    Thus, we can rewrite \refeq{eq:element} as 
    \[\jz{T}^L_{[j]}\jz{G}(i_j,r)= \sum_{i_1=1}^{I_1}\cdots \sum_{i_N=1}^{I_N} t_{i_1i_2\dots i_N}v_{i_Nj_N}\cdots v_{i_kj_k}.\]
    And when $j=k$,  $\zl{Y}\times^R_k \jz{V}_k^H = \zl{T} \times^R_N \jz{V}_N \cdots \times^R_{k-1} \jz{V}_{k-1}$. Therefore we have
    \begin{align*}
        (\jz{Y}^R_{[k]}\jz{V}_k^H)^T = \jz{T}^L_{[j]}(\jz{V}_N \otimes \cdots \otimes \jz{V}_{j+1} \otimes \jz{I}^{(j-1)} \otimes \cdots \otimes \jz{I}^{(1)}),
    \end{align*}
    namely,
    \begin{align*}
        \jz{Y}^R_{[k]} = \bigxiaokuohao{(\jz{T}^R_{[j]})^T(\jz{V}_N \otimes \cdots \otimes \jz{V}_{j+1} \otimes \jz{I}^{(j-1)} \otimes \cdots \otimes \jz{I}^{(1)})}^T\jz{V}_k.
    \end{align*}

\end{document}